\documentclass[a4paper]{amsart}
\usepackage{amsmath} 
\usepackage{amsfonts} 
\usepackage{amssymb} 
\usepackage{amsthm} 
\usepackage{thmtools} 
\usepackage[hypertexnames=false]{hyperref} 
\usepackage[utf8]{inputenc} 
\usepackage[T1]{fontenc} 
\usepackage{csquotes}
\usepackage{mathrsfs}
\usepackage[british]{babel} 
\usepackage[protrusion=true,expansion=false]{microtype} 
\usepackage{mathtools} 

	\usepackage[style=numeric]{biblatex} 
\usepackage{dsfont} 
\usepackage{mathrsfs} 
\usepackage{xcolor} 
\usepackage{lmodern} 
\usepackage{enumitem} 
\usepackage{tikz-cd}
\usepackage{setspace}
\usepackage[capitalise,noabbrev]{cleveref} 

\crefdefaultlabelformat{#2\textup{#1}#3}

\definecolor{dklimtblue}{RGB}{20,25,158}
\hypersetup{
	unicode=true,
	colorlinks=true,
	citecolor=dklimtblue,
	linkcolor=dklimtblue,
	anchorcolor=dklimtblue,
	urlcolor=dklimtblue
}

\newcommand*{\ddC}{\dot{C}}
\newcommand*{\ddF}{\dot{F}}
\newcommand*{\ddX}{\dot{X}}
\newcommand*{\ddY}{\dot{Y}}
\newcommand*{\ddZ}{\dot{Z}}
\newcommand*{\dda}{\dot{a}}
\newcommand*{\ddc}{\dot{c}}
\newcommand*{\ddf}{\dot{f}}
\newcommand*{\ddg}{\dot{g}}
\newcommand*{\ddh}{\dot{h}}
\newcommand*{\ddq}{\dot{q}}
\newcommand*{\ddu}{\dot{u}}
\newcommand*{\ddx}{\dot{x}}
\newcommand*{\ddy}{\dot{y}}

\newcommand*{\calA}{\mathcal{A}}
\newcommand*{\calC}{\mathcal{C}}
\newcommand*{\calD}{\mathcal{D}}
\newcommand*{\calE}{\mathcal{E}}
\newcommand*{\calI}{\mathcal{I}}
\newcommand*{\calM}{\mathcal{M}}

\newcommand*{\AC}{\mathsf{AC}}
\newcommand*{\AL}{\mathsf{EC}}

\newcommand*{\DC}{\mathsf{DC}}
\newcommand*{\EDC}{\mathsf{EDC}}
\newcommand*{\GCH}{\mathsf{GCH}}
\newcommand*{\PP}{\mathsf{PP}}
\newcommand*{\SVC}{\mathsf{SVC}}
\newcommand*{\WO}{\mathsf{WO}}
\newcommand*{\ZF}{\mathsf{ZF}}
\newcommand*{\ZFA}{\mathsf{ZFA}}
\newcommand*{\ZFC}{\mathsf{ZFC}}

\newcommand*{\frakJ}{\mathfrak{J}}
\newcommand*{\fraka}{\mathfrak{a}}
\newcommand*{\frakb}{\mathfrak{b}}

\newcommand*{\power}{\mathscr{P}}
\newcommand*{\sB}{\mathscr{B}}
\newcommand*{\sF}{\mathscr{F}}
\newcommand*{\sG}{\mathscr{G}}
\newcommand*{\sS}{\mathscr{S}}
\newcommand*{\sT}{\mathscr{T}}

\DeclareMathOperator{\Add}{Add}
\DeclareMathOperator{\Aut}{Aut}
\DeclareMathOperator{\Card}{Card}
\DeclareMathOperator{\FRESH}{FRESH}
\DeclareMathOperator{\HOD}{HOD}
\DeclareMathOperator{\Ord}{Ord}
\DeclareMathOperator{\Spec}{Spec}

\DeclareMathOperator{\cf}{cf}
\DeclareMathOperator{\dom}{dom}
\DeclareMathOperator{\hgt}{ht}
\DeclareMathOperator{\id}{id}
\DeclareMathOperator{\ot}{ot}
\DeclareMathOperator{\rnk}{rk}
\DeclareMathOperator{\supp}{supp}
\DeclareMathOperator{\sym}{sym}

\newcommand*{\bbB}{\mathbb{B}}
\newcommand*{\bbP}{\mathbb{P}}
\newcommand*{\bbQ}{\mathbb{Q}}
\newcommand*{\bbR}{\mathbb{R}}
\newcommand*{\bbZ}{\mathbb{Z}}

\newcommand*{\ddbbQ}{\dot{\bbQ}}

\newcommand*{\catSet}{\text{\textup{\textbf{Set}}}}

\newcommand{\1}{\mathds{1}}
\newcommand{\0}{\mathds{O}}

\DeclarePairedDelimiterX{\Set}[1]{\{}{\}}{%
	\renewcommand\mid{\SetSymbol[\delimsize]}
	#1
}
\newcommand{\SetSymbol}[1][]{%
\mathclose{}\nonscript\;#1|\nonscript\;\mathopen{}
}
\newcommand{\abs}[1]{\mathchoice
	{\left\lvert#1\right\rvert}
	{\lvert#1\rvert}
	{\lvert#1\rvert}
	{\lvert#1\rvert}
}
\newcommand{\parenth}[1]{\left(#1\right)}
\newcommand{\tup}[1]{\langle#1\rangle}

\newcommand*{\forces}{\mathrel{\Vdash}}
\newcommand*{\iter}{\mathbin{\ast}}

\newcommand*{\Inj}[2]{#2^{\underline{\smash{#1}}}} 
\newcommand*{\Surj}[2]{#1_{\overline{#2}}}
\newcommand{\res}{\nobreak\mskip2mu\mathpunct{}\nonscript
  \mkern-\thinmuskip{\upharpoonright}\mskip6muplus1mu\relax} 

\newcommand{\comp}{\mathrel{\|}}
\newcommand{\concat}{\mathbin{\raisebox{.9ex}{\scalebox{.7}{\(\frown\)}}}}
\newcommand*{\defeq}{\mathrel{\vcenter{\baselineskip0.5ex \lineskiplimit0pt
                     \hbox{\scriptsize.}\hbox{\scriptsize.}}}%
                     =}
\newcommand{\lomega}{{{<}\omega}}
\newcommand{\lra}{\longleftrightarrow}
\newcommand{\varep}{\varepsilon}
\newcommand{\vphi}{\varphi}

\setlist[enumerate,1]{label=\textup{\arabic*.},ref=\textup{(\arabic*)}}
\setlist[enumerate,2]{label=\textup{(\roman*)},ref=\textup{(\roman*)}}

\newtheoremstyle{boldrk}
  {}{}
  {}{}
  {\bfseries}{.}
  {5pt plus 1pt minus 1pt}{}

\theoremstyle{plain}

\newtheorem{thm}{Theorem}[section]
\newtheorem{claim}{Claim}[thm]
\newtheorem{cor}[thm]{Corollary}
\newtheorem{lem}[thm]{Lemma}
\newtheorem{prop}[thm]{Proposition}
\newtheorem*{fact}{Fact}

\theoremstyle{boldrk}

\newtheorem{qn}[thm]{Question}
\newtheorem*{qnast}{Question}

\theoremstyle{definition}

\newtheorem{defn}[thm]{Definition}

\newenvironment*{poc}[1][Proof of Claim]{\begin{proof}[#1]}{\end{proof}}

\Crefname{thm}{Theorem}{Theorems}
\Crefname{claim}{Claim}{Claims}
\Crefname{cor}{Corollary}{Corollaries}
\Crefname{lem}{Lemma}{Lemmas}
\Crefname{prop}{Proposition}{Propositions}
\Crefname{fact}{Fact}{Facts}
\Crefname{thmalph}{Theorem}{Theorems}
\Crefname{qn}{Question}{Questions}
\Crefname{defn}{Definition}{Definitions}
\Crefname{conj}{Conjecture}{Conjectures}

\title[Eccentricity, extendable choice and descending distributivity]{Eccentricity, extendable choice and\\descending distributive forcing}
\author{Calliope Ryan-Smith}

\email{c.Ryan-Smith@leeds.ac.uk}
\urladdr{https://academic.calliope.mx}
\address{School of Mathematics, University of Leeds, LS2 9JT, UK}

\date{13th June 2025}
\thanks{The author's work was financially supported by EPSRC via the Mathematical Sciences Doctoral Training Partnership [EP/W523860/1]. For the purpose of open access, the author has applied a Creative Commons Attribution (CC BY) licence to any Author Accepted Manuscript version arising from this submission. No data are associated with this article. All translations are by the author and intended to be faithful to the mathematical content, rather than literal.}

\begin{document}

\begin{abstract}
We introduce the forcing property of \emph{descending distributivity}. A forcing \(\bbP\) is \emph{\(\kappa\)-descending distributive} if for all decreasing sequences \((D_\alpha)_{\alpha < \kappa}\) of open dense sets, \(\bigcap_\alpha D_\alpha\) is open dense. This generalises the informal idea \(\bbP\) doesn't affect much on the scale of \(\kappa\), such as if \(\bbP\) is \(\kappa\)-distributive or if \(\kappa > \abs{\bbP}\). For example, a \(\kappa\)-descending distributive forcing will not change the cofinality of \(\kappa\) or introduce fresh functions on \(\kappa\).

Using this, we investigate the phenomena of eccentric sets, those sets \(X\) such that, for some ordinal \(\alpha\), \(X\) surjects onto \(\alpha\), but \(\alpha\) does not inject into \(X\). We refine prior works of the author by giving explicit calculations for the Hartogs and Lindenbaum numbers in eccentric constructions and providing a sharper description of the Hartogs--Lindenbaum spectra of models of small violations of choice. To do so we further develop an axiom (scheme) introduced by Levy that we call the axiom of \emph{extendable choice}. For an ordinal \(\alpha\), \(\AL_\alpha\) asserts that if \(\emptyset\notin A=\Set{A_\gamma \mid \gamma < \alpha}\) and, for all \(\beta<\alpha\), \(\Set{A_\gamma \mid \gamma < \beta}\) has a choice function, then \(A\) has a choice function. This is closely tied to the presence of eccentric sets, and we construct symmetric extensions that give fine control over the \(\alpha\) for which \(\AL_\alpha\) holds by using descending distributivity.
\end{abstract}

\maketitle

\tableofcontents

\section{Introduction}\label{s:introduction}

The Cantor--Bernstein theorem (\cref{thm:zweidrittelsatz}) is a fascinating and fundamental theorem in cardinal arithmetic, proving right a highly desirable intuition for cardinal numbers. While being an immediate consequence of the well-ordering theorem in \(\ZFC\), the fact that \cref{thm:zweidrittelsatz} remains true even in \(\ZF\) is a great relief compared to the wild behaviour that choiceless models can exhibit.\footnote{Gitik's model from \cite{gitik_uncountable_1980} in which all uncountable cardinals are singular is a great example of an extremely wild model. Not only does \(\AC\) fail (dramatically), one cannot recover \(\AC\) even with class forcing, or by any means short of adding new ordinals to the universe.}

\begin{thm}\label{thm:zweidrittelsatz}
If \(\abs{M} \leq \abs{N}\) (that is, there is an injection \(M \to N\)) and \(\abs{N} \leq \abs{M}\) then \(\abs{M} = \abs{N}\).
\end{thm} 

While \cref{thm:zweidrittelsatz} was originally stated in \cite{cantor_mitteilungen_1887} without proof, Cantor would later prove the statement as a corollary of the claim that cardinal numbers are linearly ordered, about which he said ``only later, when we better understand transfinite cardinal numbers, will the truth become clear: for all cardinal numbers \(\fraka\) and \(\frakb\), either \(\fraka = \frakb\), \(\fraka < \frakb\), or \(\fraka > \frakb\)'' \cite[p.~285]{cantor_gesammelte_1932} (reprinting \cite{cantor_beitrage_i_1895,cantor_beitrage_ii_1897}).\footnote{Note that Cantor did \emph{not} assert that this claim is obvious. Indeed, just prior he had written ``it is \emph{by no means self evident} [emphasis added] \ldots{} one of these relations should hold'' \cite[p.~285]{cantor_gesammelte_1932}.} Cantor ultimately did not supply any \(\ZF\)-viable argument that the cardinal numbers could be linearly ordered in this way and, in fact, this strategy was doomed to fail.

In \cite{hartogs_uber_1915} Hartogs supplies us with an elegant argument. Suppose that any two cardinals are comparable, so for all sets \(M\) and \(N\), \(\abs{M}=\abs{N}\), \(\abs{N}>\abs{M}\), or \(\abs{M} < \abs{N}\). Now suppose that, given some set \(M\), there is an \emph{ordinal} \(\alpha\) such that \(\abs{\alpha}\nless\abs{M}\). By our assumption we have that \(\abs{M}\leq\abs{\alpha}\) instead and so \(M\) can be well-ordered.

\begin{thm}[Hartogs]\label{thm:hartogs-lemma}
``For all sets \(M\) there is a well-orderable set \(L\) such that \(\abs{L}\nless\abs{M}\)'' \textup{\cite[p.~442]{hartogs_uber_1915}.}\footnote{A proof in English of \cref{thm:hartogs-lemma} can be found in \cite[Theorem~8.18]{goldrei_classic_1996}.}
\end{thm}

\begin{cor}[Hartogs]\label{cor:hartogs-theorem}
``Each of the three principles, firstly the axiom of choice, secondly the comparability of sets and thirdly the well-ordering principle must be considered equivalent'' \textup{\cite[p.~438]{hartogs_uber_1915}.}
\end{cor}

Doomed indeed. It is perhaps ironic that the technology used for Hartogs's proof, this well-orderable set \(L\), has become the more important object. We have even gone so far as to give \cref{thm:hartogs-lemma} (often called \emph{Hartogs's lemma}) the mantle of \emph{theorem} while \cref{cor:hartogs-theorem} is relegated to \emph{corollary}.\footnote{Compare Ramsey's publication that introduces Ramsey's theorem, undersold somewhat with the simple abstract ``this paper is primarily concerned with a special case of one of the leading problems of mathematical logic . \ldots{} In the course of this investigation it is necessary to use certain theorems on combinations which have an independent interest'' \cite[p.~264]{ramsey_problem_1929}.} We now refer to the least ordinal \(\alpha\) such that a given set \(M\) admits no injection \(\alpha\to M\) as the \emph{Hartogs number} of \(M\), denoted \(\aleph(M)\).\footnote{\(H(M)\) also sees use, as do other notations in all likelihood.} It is straightforward to conclude that \(\aleph(M)\) must also be a cardinal, in the sense that \(\abs{\aleph(M)}\neq\abs{\beta}\) for any \(\beta<\aleph(M)\).

After this, Lindenbaum and Tarski released \emph{Communication sur les recherches de la th\'eorie des ensembles} \cite{lindenbaum_communication_1926}, an announcement of myriad results in set theory that one author or the other had proved and presented at various meetings. Included within was not only much discussion of the relation \(\leq\), but also its dual \(\leq^\ast\) (denoted \({\leq}\ast\) in \cite{lindenbaum_communication_1926}). We say that \(\abs{M}\leq^\ast\abs{N}\) if there is a surjection \(N \to M\) or if \(M=\emptyset\). One result in particular is very familiar:

\begin{thm}[Lindenbaum]
``The axiom of choice is equivalent to the statement that . \ldots{} for all \(M\) and \(N\), \(\abs{M}\leq^\ast\abs{N}\) or \(\abs{N}\leq^\ast\abs{M}\)'' \textup{\cite[Th\'eor\`eme~82.A\({}_6\), pp.~185--186]{lindenbaum_communication_1926}.}
\end{thm}

While this theorem is supplied without proof, the rest of the paper gives the reader enough direction to realise what has happened. Lindenbaum has shown that, for all \(M\), there is an ordinal \(\alpha\) such that \(\abs{\alpha}\nleq^\ast\abs{M}\). For a minimal proof, one could refer to the following theorem of Tarski: ``if \(\abs{M}\leq^\ast\abs{N}\) then \(\abs{\power(M)}\leq\abs{\power(N)}\)'' \cite[Th\'eor\`eme~4, p.~175]{lindenbaum_communication_1926}, followed by observing that \(\abs{M}\leq\abs{\power(M)}\).\footnote{The first explicit published proof of the existence of Lindenbaum numbers is found in \cite{sierpinski_sur-lindenbaum_1947}.} We refer to the least such ordinal as the \emph{Lindenbaum number} of \(M\), denoted \(\aleph^\ast(M)\). As with the Hartogs number, \(\aleph^\ast(M)\) will also be a cardinal. Furthermore, because \(\abs{M}\leq\abs{N}\implies\abs{M}\leq^\ast\abs{N}\), we also have that \(\aleph(M)\leq\aleph^\ast(M)\).

This all becomes trivial if one assumes the axiom of choice since for all \(M\) there is a least ordinal in bijection with \(M\) (denoted \(\abs{M}\)) and, in this case, \(\aleph(M)\) and \(\aleph^\ast(M)\) are both equal to \(\abs{M}^+\). Thus the Hartogs and Lindenbaum numbers tell us only about the cardinality of the set itself. On the other hand, it seems as though in many models of \(\ZF+\lnot\AC\) the statement \((\forall M)\aleph(M)=\aleph^\ast(M)\) (`\(\aleph=\aleph^\ast\)', to be pithy) fails. For example, in Cohen's first model of \(\lnot\AC\) from \cite{cohen_independence_1963}, there is an infinite set \(A\) of real numbers such that \(\aleph(A)=\aleph_0\) but \(\aleph^\ast(A)=\aleph_1\). A natural question arises:

\begin{qnast}
Does \(\aleph=\aleph^\ast\) imply the axiom of choice?
\end{qnast}

This was answered by a lemma in \cite{pelc_weak_1978} that the author attributes to Pincus.

\begin{thm}[{Pincus \cite{pelc_weak_1978}}]\label{thm:pincus}
\(\aleph=\aleph^\ast\) is equivalent to \(\AC_\WO\), the axiom of choice for well-orderable families.\footnote{In fact, the statement of the lemma in \cite{pelc_weak_1978} is much weaker than \cref{thm:pincus}. The original statement is more accurately presented as `the partition principle implies \(\AC_\WO\)'. However, a quick observation notes that the proof only uses that Hartogs and Lindenbaum numbers are the same, not the full strength of the partition principle. The reverse direction is very straightforward, so \cref{thm:pincus} is attributed entirely to Pincus.}
\end{thm}

As \(\AC_\WO\) does not imply \(\AC\), also proved in \(\ZF\) by a construction of Pincus \cite{pincus_individuals_1969}, we obtain that \(\aleph=\aleph^\ast\) is in fact weaker than \(\AC\).

We find \cref{thm:pincus} to be a very compelling result as it relates to the partition problem. The \emph{partition principle} \(\PP\) is the statement `for all \(X,Y\), if \(\abs{X}\leq^\ast\abs{Y}\) then \(\abs{X}\leq\abs{Y}\),' and the \emph{partition problem} is the question `does \(\PP\) imply \(\AC\)?' While this is still open, \cref{thm:pincus} demonstrates that `\(\PP\) for well-ordered sets' implies `\(\AC\) for well-ordered sets.' To be more precise, the statement `for all \(X,Y\), if \(\abs{X}\leq^\ast\abs{Y}\) \emph{and at least one of \(X\) or \(Y\) can be well-ordered} then \(\abs{X}\leq\abs{Y}\)' implies the statement `for all \(X\), if \(\emptyset\notin X\) \emph{and \(X\) can be well-ordered}, then \(X\) has a choice function.'

The key mechanism behind \cref{thm:pincus} is a method of taking a well-ordered set \(A\) and extrapolating from this a new set \(X\) such that, for some large enough ordinal \(\alpha\), an injection \(\alpha\to X\) could be translated into a choice function for \(A\). \(X\) is further constructed so that \(\aleph^\ast(X)>\alpha\) and thus, assuming \(\aleph(X)=\aleph^\ast(X)\), one obtains the injection \(\alpha\to X\) as desired. By refining this construction, Ryan-Smith \cite{ryan-smith_acwo_2024} (author) presented this technique as a way of transforming \emph{eccentric sets}---those sets \(X\) with \(\aleph(X)<\aleph^\ast(X)\)---into eccentric sets of larger Hartogs and Lindenbaum numbers. As a corollary, they presented several more equivalent statements to \(\AC_\WO\), including `for all \(X\), \(\aleph(X)\) is regular' and `for all \(X\), \(\aleph(X)\) is a successor.' The constructions were also used to show a regularity of the \emph{Hartogs--Lindenbaum spectrum} of models of \(\SVC\).\footnote{The \emph{Hartogs--Lindenbaum spectrum} of \(M \models \ZF\) is \(\Spec_\aleph(M) = \Set{\tup{\aleph(X),\aleph^\ast(X)} \mid X \in M}\). See \cref{s:hartlin} for further details.} We shall further refine these constructions, omitting certain arbitrary requirements and giving more clear bounds and calculations.

\begin{cor}[\cref{cor:svc-bounds-on-eccentricity}]
Assume \(\SVC(S)\). If \(\aleph(X)\leq\lambda<\aleph^\ast(X)\) then either:
\begin{enumerate}
\item \(\aleph^\ast(X)\leq\aleph^\ast(S)\); or
\item \(\lambda=\aleph(X)<\aleph^\ast(X)=\lambda^+\) and \(\cf(\lambda)<\aleph^\ast(S)\).
\end{enumerate}
\end{cor}

\begin{thm}[{\cref{thm:going-up-regular,thm:going-up-singular-or-limit}}]
Suppose that \(\aleph(X)\leq\cf(\lambda)<\aleph^\ast(X)\). Then there is a set \(Y\) with \(\aleph(Y)=\lambda\) and \(\aleph^\ast(Y)=\aleph^\ast(X)\times\lambda^+\).

Suppose that \(\aleph(X)\) is singular or a limit, \(\lambda\geq\sup\Set{\aleph(\alpha^\alpha)\mid\alpha<\aleph(X)}\) and \(\cf(\lambda)=\cf(\aleph(X))\). Then there is a set \(Y\) with \(\aleph(Y)=\lambda\) and \(\aleph^\ast(Y)=\aleph^\ast(\Inj{{<}\mu}{X})\times\lambda^+\).
\end{thm}

This provides a positive answer to a question posed by Ryan-Smith: ``let \(\mu\) be weakly inaccessible and suppose that for some set \(X\), \(\aleph(X)<\mu<\aleph^\ast(X)\). Must there exist \(Y\) such that \(\aleph(Y)=\mu\)?'' \cite[Question~5.3, p.~222]{ryan-smith_acwo_2024}

We shall also recontextualise such findings as they relate to Levy's paper \cite{levy_interdependence_1964} in which Levy devotes time to a scheme of choice principles closely related to eccentricity. While no name was originally given for this axiom that Levy denoted \(C(\alpha)\) (for \(\alpha\) an ordinal), we call it the axiom of \emph{extendable choice}, denoted \(\AL_{\alpha}\):\footnote{We believe that using the original notation of \(C(\alpha)\) has the potential for confusion. Indeed, in \cite{levy_interdependence_1964}, Levy denotes by \(C^\ast(\alpha)\) the axiom that we would write as \(\AC_\alpha\) today.}

\begin{displaycquote}[p.~136]{levy_interdependence_1964}
Let \(F\) be a function on \(\alpha\) such that, for each \(\beta<\alpha\), \(F(\beta)\) is a non-void set. If for every \(\gamma<\alpha\) there exists a function \(H\) on \(\gamma\) such that \(H(\beta)\in F(\beta)\) for \(\beta<\gamma\), then there exists a function \(G\) on \(\alpha\) such that \(G(\beta)\in F(\beta)\) for \(\beta<\alpha\).
\end{displaycquote}

That is, if \(X=\Set{A_\beta\mid\beta<\alpha}\) has choice functions for all `initial segments' \(\Set{A_\beta\mid\beta<\gamma}\) then in fact \(X\) has a choice function. Extendable choice represents a highly local form of Pincus's \cref{thm:pincus}, working on the level of a single cofinality.

\begin{cor}[{\cref{cor:al-equivalents}}]
Let \(\kappa\) be an infinite cardinal and \(\calC\) be the class of cardinals \(\mu\) with \(\cf(\mu)=\cf(\kappa)\). The following are equivalent:
\begin{enumerate}
\item \(\AL_{\kappa}\);
\item for all \(\mu\in\calC\), \(\AL_{\mu}\);
\item for some \(\mu\in\calC\), \(\AL_{\mu}\);
\item there is a limit \(\mu\in\calC\) and a set \(X\) such that \(\aleph(X)=\mu\);
\item there is \(\mu\in\calC\) and a set \(X\) such that \(\aleph(X)=\mu<\aleph^\ast(X)\); and
\item there is \(\mu^\ast\in\calC\) such that for all limits \(\mu\in\calC\setminus\mu^\ast\) there is a set \(X\) such that \(\aleph(X)=\mu\).
\end{enumerate}
\end{cor}

As part of his analysis, Levy demonstrates the following implications between these principles.

\begin{thm}[{\cite[Theorems~1 and 2]{levy_interdependence_1964}}]
For all ordinals \(\alpha\) and \(\beta\):
\begin{enumerate}
\item \(\AL_{0}\);
\item \(\AL_{\alpha+1}\); and
\item if \(\cf(\alpha)=\cf(\beta)\) then \(\AL_{\alpha}\lra\AL_{\beta}\).
\end{enumerate}
\end{thm}

Levy then goes on to prove in \(\ZFA\) that no further \(\AL\) relations exist. In \cref{s:levys-axiom;ss:violating-al} we demonstrate this in \(\ZF\) using a variant of the construction in \cite{karagila_which_2024}.

\begin{thm}[{\cref{thm:new-models}}]
Assume \(\GCH\) and let \(\calC\) be a class of infinite regular cardinals such that, whenever \(\kappa\) is inaccessible and \(\kappa \cap \calC\) is unbounded below \(\kappa\), we have \(\kappa \in \calC\). Then there is a symmetric extension \(M\) such that
\begin{equation*}
\Spec_\aleph(M)=\Set*{\tup{\kappa^+,\kappa^+}\mid\kappa\in\Card}\cup\Set*{\tup{\kappa,\kappa^+}\mid\cf(\kappa)\in\calC}.
\end{equation*}
In particular, \(M\models(\forall\kappa)\AL_{\kappa}\lra\cf(\kappa)\notin\calC\).
\end{thm}

To do so, we produce some preservation theorems for \(\AL\), which leads us down the path of defining \emph{descending distributivity}. A forcing \(\bbP\) is \emph{\(\kappa\)-descending distributive} if for all decreasing sequences \(\tup{D_\alpha \mid \alpha < \kappa}\) of open dense subsets of \(\bbP\), \(\bigcap_{\alpha < \kappa} D_\alpha\) is open dense. This generalises two ways in which a forcing \(\bbP\) may be thought of as `not affecting much of the structure of \(\kappa\)': if \(\kappa \geq \aleph^\ast( \bbP )\) is regular, or \(\bbP\) is adds no \(\kappa\)-sequences, then \(\bbP\) is \(\kappa\)-descending distributive.

\begin{prop}[{\cref{prop:generous-doesnt-change-cf}}]
If \(\bbP\) is \(\kappa\)-descending distributive then \(\bbP\) does not change the cofinality of \(\kappa\).
\end{prop}

We also extend the notion to symmetric systems by defining descending distributivity for filters of partial orders. In particular, a filter \(\sF\) of subgroups of \(\sG\) is \(\kappa\)-descending distributive if for all sequences \(\tup{H_\alpha \mid \alpha < \kappa}\) of elements of \(\sF\) there is unbounded \(I \subseteq \kappa\) such that \(\bigcap_{\alpha \in I} H_\alpha \in \sF\). A symmetric system \(\sS = \tup{\bbP, \sG, \sF}\) is \(\kappa\)-descending distributive if both \(\bbP\) and \(\sF\) are.

\begin{thm}[\cref{thm:al-preservation,thm:al-non-preservation}]
Let \(\kappa\) be a regular cardinal and \(\sS\) be \(\kappa\)-descending distributive. Then \(\sS\) cannot change the truth value of \(\AL_\kappa\).
\end{thm}

An application of note is in \emph{fresh sequences}. For \(V \models \ZF\), a function \(f \colon \kappa \to V\) is \emph{fresh} if \(f \notin V\) but, for all \(\alpha < \kappa\), \(f \res \alpha \in V\). These were originally introduced by Hamkins in \cite{hamkins_gap_2001}. While no true `internal' classification of when a forcing adds a no fresh functions \(\kappa \to V\) exists, deep exploration has taken place regarding the patterns of fresh functions that may be present (such as \cite{fischer_fresh_2023}). Descending distributivity, too, has a lot to say about fresh functions.

\begin{thm}[{\cref{thm:generous-implies-no-fresh}}]
If \(\bbP\) is \(\kappa\)-descending distributive then it adds no fresh functions \(\kappa \to V\).

Suppose instead that for all \(V\)-generic filters \(G \subseteq \bbP\) and all open dense sets \(D \subseteq \bbP\) (with \(D \in V\)), that \(G \setminus D \in V\). Then \(\bbP\) is \(\kappa\)-descending distributive \emph{if and only if} it adds no fresh functions \(\kappa \to V\).
\end{thm}

\subsection{Structure of the paper}

In \cref{s:preliminaries} we lay out preliminaries for the paper, particularly Hartogs and Lindenbaum numbers, symmetric extensions and small violations of choice. In \cref{s:generosity} we discuss the property of descending distributivity for a notion of forcing or a symmetric extension and how it relates to fresh functions. In \cref{s:levys-axiom} we discuss extendable choice, both historic results by Levy as well as new preservation theorems and constructions for controlling extendable choice. In \cref{s:hartlin} we give more refined constructions for eccentric sets and bounds on Hartogs--Lindenbaum spectra for models of \(\SVC\). In \cref{s:future} we re-state open questions asked throughout the paper.

\section{Preliminaries}\label{s:preliminaries}

We work in \(\ZF\). For sets \(X,Y\), we write \(\abs{X}\leq\abs{Y}\) to mean that there is an injection \(X\to Y\), \(\abs{X}\leq^\ast\abs{Y}\) to mean that there is a surjection \(Y\to X\) or \(X=\emptyset\), and \(\abs{X}=\abs{Y}\) to mean that there is a bijection \(X\to Y\). When \(X\) is well-orderable, we denote by \(\abs{X}\) its cardinality, the least ordinal \(\alpha\) such that \(\abs{X}=\abs{\alpha}\). Note that we do not use this notation when \(X\) cannot be well-ordered. We denote the class of ordinals by \(\Ord\). By a \emph{cardinal} we always mean a well-orderable cardinal; that is, \(\alpha\in\Ord\) such that \(\abs{\alpha}=\alpha\). We denote by \(\Inj{X}{Y}\) the set of injections \(X \to Y\). For \(A\) a set of ordinals, we denote by \(\ot(A)\) its \emph{order type}, the unique ordinal \(\alpha\) such that \(\tup{A,{\in}} \cong \tup{\alpha,{\in}}\).

Given a set \(X\), a \emph{choice function} for \(X\) is a function \(f\colon X\to\bigcup X\) such that, for all \(x\in X\), \(f(x)\in x\). Given a function \(F\colon A\to B\), a \emph{selector} for \(F\) is a function \(f\colon A\to\bigcup B\) such that, for all \(a\in A\), \(f(a)\in F(A)\).

A \emph{tree} is a partially ordered set \(\tup{T,{\leq}}\) with least element such that, for all \(t\in T\), \(\Set{s\in T\mid s\leq t}\) is well-ordered by \(\leq\). This gives rise to a notion of rank \(\rnk_T(t)=\ot\Set{s\in T\mid s<t}\), levels \(T_\alpha=\Set{t\in T\mid \rnk(t)=\alpha}\) and height \(\hgt(T)=\sup\Set{\rnk_T(t)+1\mid t\in T}\). \(T\) is \emph{\(\lambda\)-closed} if every chain of order type less than \(\lambda\) in \(T\) has an upper bound. The \emph{principle of dependent choice} (for \(\lambda\)), denoted \(\DC_\lambda\), is the assertion that every \(\lambda\)-closed tree has a maximal node or a chain of order type \(\lambda\). Note that if \(\lambda\) is regular and \(T\) is of height \(\lambda\) then \(\DC_\lambda\) asserts that \(T\) has a maximal node or a \emph{branch}, a downwards-closed chain \(b\subseteq T\) intersecting every level of \(T\).

\subsection{Hartogs and Lindenbaum numbers}

For a set \(X\), the \emph{Hartogs number} of \(X\), denoted \(\aleph(X)\), is defined as \(\min\Set{\alpha\in\Ord\mid\abs{\alpha}\nleq\abs{X}}\). The \emph{Lindenbaum number} of \(X\), denoted \(\aleph^\ast(X)\), is defined as \(\min\Set{\alpha\in\Ord\mid\abs{\alpha}\nleq^\ast\abs{X}}\). For all \(X\) the Hartogs and Lindenbaum numbers are well-defined cardinals such that \(\aleph(X)\leq\aleph^\ast(X)\).

\begin{thm}[Diener, Tarski, folklore]
For all infinite \(X\) and all non-empty \(Y\), the following hold:
\begin{align}
\aleph(X+Y)&=\aleph(X)+\aleph(Y)\label{item:hartlin-a}\\
\aleph(X\times Y)&=\aleph(X)\times\aleph(Y)\label{item:hartlin-b}\\
\aleph^\ast(X+Y)&=\aleph^\ast(X)+\aleph^\ast(Y)\label{item:hartlin-c}\\
\aleph^\ast(X)\times\aleph^\ast(Y) \leq \aleph^\ast(X\times Y) &\leq (\aleph^\ast(X)\times\aleph^\ast(Y))^+\label{item:weak-productivity}
\end{align}
\end{thm}

\begin{proof}
\cref{item:hartlin-a,item:hartlin-b} are due to Tarski \cite[Th\'eor\`eme~76]{lindenbaum_communication_1926}. \cref{item:hartlin-c} is a folklore result or corollary of Tarski \cite[Th\'eor\`eme~60]{lindenbaum_communication_1926}. \cref{item:weak-productivity} is a folklore result or corollary of Diener \cite[Theorem~2.1]{diener_transitive_1992}; we give a direct proof in \cref{thm:weak-productivity}.
\end{proof}

We refer to \cref{item:weak-productivity} as the \emph{weak productivity} of Lindenbaum numbers.

\begin{thm}[Folklore]\label{thm:weak-productivity}
If \(X\) is infinite and \(Y\) is non-empty then
\begin{equation*}
\aleph^\ast(X)\times\aleph^\ast(Y)\leq\aleph^\ast(X\times Y)\leq(\aleph^\ast(X)\times\aleph^\ast(Y))^+.
\end{equation*}
\end{thm}

\begin{proof}
\(\aleph^\ast(X)\times\aleph^\ast(Y)\leq\aleph^\ast(X\times Y)\) is straightforward, so we shall present only the other bound.

Let \(f\) be a surjection \(X\times Y\to\lambda\). For \(x\in X\), define \(r_x=f``(\Set{x}\times Y)\), \(\gamma_x=\ot(r_x)<\aleph^\ast(Y)\) and \(\gamma=\sup\Set{\gamma_x\mid x\in X}\leq\aleph^\ast(Y)\). Then we induce a function \(g\colon X\times\gamma\to\lambda\), where \(g(x,\alpha)\) is the \(\alpha\)\textup{th} element of \(r_x\) (if defined). Note that \(g``X\times\gamma=f``X\times Y=\lambda\). For \(\alpha<\gamma\), let \(l_\alpha=g``X\times\Set{\alpha}\), \(\delta_\alpha=\ot(l_\alpha)<\aleph^\ast(X)\) and \(\delta=\sup\Set{\delta_\alpha\mid\alpha<\gamma}\leq\aleph^\ast(X)\). Then we induce a function \(h\colon\delta\times\gamma\to\lambda\), where \(h(\varep,\alpha)\) is the \(\varep\)\textup{th} element of \(l_\alpha\) (if defined). Again \(h``\delta\times\gamma=g``X\times\gamma=\lambda\), so \(\abs{\lambda}\leq^\ast\abs{\delta\times\gamma}\) and \(\abs{\lambda}\leq\abs{\delta\times\gamma}\leq\abs{\aleph^\ast(X)\times\aleph^\ast(Y)}\). Thus \(\lambda<(\aleph^\ast(X)\times\aleph^\ast(Y))^+\) as required.
\end{proof}

The weak productivity of Lindenbaum numbers cannot be replaced by productivity in general. For example, in the Feferman--Levy model (in which \(\bbR\) is a countable union of countable sets, \cite[Model~\(\calM9\)]{howard_consequences_1998})\footnote{Announced in \cite{feferman_independence_1963}. Cohen \cite[p.~143]{cohen_set_1966} is an early work demonstrating this model.} there is a set \(X\) such that \(\aleph(X)=\aleph^\ast(X)=\aleph_1\) but \(\abs{X^2}\leq^\ast\abs{\power(X)}\leq^\ast\abs{\bbR}\leq^\ast\abs{X^2}\). In particular, \(\aleph^\ast(X^2)=\aleph_2\) (result due to Glazer \cite{glazer_mo_2023}).

However, unpublished work of Peng (private communication) shows that Lindenbaum numbers do obey \emph{well-ordered} productivity.
\begin{thm}[Peng]\label{thm:lindenbaum-wo-productivity}
If \(\kappa\) is a cardinal and \(X\) is an infinite set then
\begin{equation*}
\aleph^\ast(X\times\kappa)=\aleph^\ast(X)\times\kappa^+.
\end{equation*}
\end{thm}
Note in this case that \(\aleph^\ast(X)\times\kappa^+=\aleph^\ast(X)\times\aleph^\ast(\kappa)\), as expected. A proof for the case \(\kappa=\aleph_0\) can be found in Peng--Shen \cite[Lemma~3.6]{peng_generalized_2024}.

\subsection{Forcing and symmetric extensions}

Our treatment of forcing is fairly standard, similar to that found in \cite[Chapter~14]{jech_set_2003}. A \emph{notion of forcing} (or simply \emph{forcing}) is a partial order \(\tup{\bbP,{\leq}}\) with maximal element, denoted \(\1\). Elements of \(\bbP\) are called \emph{conditions} and we force downwards, so we say that a condition \(q\) is \emph{stronger} than \(p\) (or \emph{extends} \(p\)) to mean that \(q\leq p\). Two conditions \(p,q\in\bbP\) are said to be \emph{compatible}, written \(p\comp q\), if there is \(r\in\bbP\) such that \(r\leq p\) and \(r\leq q\). If \(p\) and \(q\) are not compatible then we call them \emph{incompatible}, denoted \(p\perp q\). By a \emph{\(\bbP\)-name}, we mean a set \(\ddx\) with elements only of the form \(\tup{p,\ddy}\) for \(p\in\bbP\) and \(\ddy\) a \(\bbP\)-name. We denote the class of \(\bbP\)-names by \(V^\bbP\). Given a collection \(X=\Set{\ddx_i\mid i\in I}\) of \(\bbP\)-names we denote by \(X^\bullet\) or \(\Set{\ddx_i\mid i\in I}^\bullet\) the name \(\Set{\tup{\1_\bbP,\ddx_i}\mid i\in I}\). We inductively define the \emph{check name} of a set \(x\) as \(\check{x}=\Set{\check{y}\mid y\in x}^\bullet\). A set \(G\subseteq\bbP\) is a \emph{\(V\)-generic filter} if: (1) for all \(p,q\in G\) there is \(r\in G\) extending \(p\) and \(q\); (2) if \(p\geq q\in G\) then \(p\in G\); and (3) for all open dense \(D\subseteq\bbP\) in \(V\), \(G\cap D\neq\emptyset\). Given \(V\)-generic \(G\subseteq\bbP\) we interpret a \(\bbP\)-name \(\ddx\) inductively as \(\ddx^G=\Set{\ddy^G\mid(\exists p\in G)\tup{p,\ddy}\in\ddx}\). Note that in this case \((X^\bullet)^G=\Set{\ddx^G\mid \ddx\in X}\) and \(\check{x}^G=x\). We extend the bullet notation to tuples following this context, setting \(\tup{\ddx,\ddy}^\bullet\) to be the canonical name for the tuple \(\tup{\ddx^G,\ddy^G}\) in the extension. The forcing extension by \(V\)-generic \(G\subseteq\bbP\), denoted \(V[G]\), is the class \(\Set{\ddx^G\mid\ddx \in V^\bbP}\). We denote the \emph{forcing relation} for \(\bbP\) by \(\forces_\bbP\), acting as a relation between \(\bbP\) and formulae in the language of set theory with \(\bbP\)-names for free variables.

\subsubsection{Symmetric extensions}

A \emph{symmetric system} is a triple \(\sS=\tup{\bbP,\sG,\sF}\), where \(\bbP\) is a forcing, \(\sG\leq\Aut(\bbP)\) and \(\sF\) is a normal filter of subgroups of \(\sG\). That is, \(\sF\) is a set of subgroups of \(\sG\) such that:
\begin{enumerate}
\item if \(H\geq H'\in\sF\) then \(H\in\sF\);
\item \(\sG\in\sF\);
\item if \(H,H'\in\sF\) then \(H\cap H'\in\sF\); and
\item if \(H\in\sF\) and \(\pi\in\sG\) then \(\pi H\pi^{-1}\in\sF\).
\end{enumerate}
For a cardinal \(\lambda\) we say that a filter of subgroups \(\sF\) is \emph{\(\lambda\)-complete} if for all \(E \in [\sF]^{{<}\lambda}\), \(\bigcap E \in \sF\). We extend the action of \(\sG\) to all \(\bbP\)-names inductively, defining \(\pi\ddx=\Set{\tup{\pi p,\pi\ddy}\mid\tup{p,\ddy}\in\ddx}\). Such automorphisms give rise to the following symmetry lemma.

\begin{lem}[{The Symmetry Lemma, \cite[Lemma~14.37]{jech_set_2003}}]\label{lem:symmetry}
Let \(\bbP\) be a forcing, \(p\in\bbP\), \(\pi\in\Aut(\bbP)\), \(\ddx\) a \(\bbP\)-name and \(\vphi(x)\) a formula. Then \(p\forces_\bbP\vphi(\ddx)\) if and only if \(\pi p\forces_\bbP\vphi(\pi\ddx)\).
\end{lem}

A \(\bbP\)-name \(\ddx\) is \emph{\(\sF\)-symmetric} if \(\sym(\ddx)=\Set{\pi\in\sG\mid\pi\ddx=\ddx}\in\sF\) and it is an \emph{\(\sS\)-name} (or \emph{hereditarily \(\sF\)-symmetric}) if it is \(\sF\)-symmetric and, for all \(\tup{p,\ddy}\in\ddx\), \(\ddy\) is an \(\sS\)-name. We denote the class of \(\sS\)-names by \(V^\sS\). By relativising the forcing relation to the class of \(\sS\)-names we produce a new forcing relation \(\forces_\sS\). This forcing relation obeys the analogue of \cref{lem:symmetry} for \(\pi\in\sG\). A set \(G\) is \emph{\(\sS\)-generic} if it is a \(V\)-generic filter on \(\bbP\).\footnote{This terminology is not standard and upcoming works use `\(\sS\)-generic' to refer to a generalisation of \(\bbP\)-generic filters specific to the symmetric system \(\sS\).} The \emph{symmetric extension} generated by \(\sS\) and \(G\), denoted \(V[G]_\sS\), is the class \(\Set{\ddx^G\mid\ddx \in V^\sS}\).

\begin{thm}[{\cite[Lemma~15.51]{jech_set_2003}}]
Let \(\sS=\tup{\bbP,\sG,\sF}\) be a symmetric system and \(G\) an \(\sS\)-generic filter in \(V \models \ZF\). Then \(V\subseteq V[G]_\sS \subseteq V[G]\) are transitive in \(V[G]\) and are all models of \(\ZF\).
\end{thm}

The following technique appears in \cite{holy_ordering_2025}.

\begin{fact}[Symmetric mixing]
Let \(\sS=\tup{\bbP,\sG,\sF}\) be a symmetric system and \(\vphi(u,v)\) be a formula in the language of set theory. Then there is a definable class function \(F\) so that for all \(\sS\)-names \(\ddx\) and conditions \(p\in\bbP\), if \(p\forces_\sS(\exists!y)\vphi(y,\ddx)\) then \(F(p,\ddx)\) is an \(\sS\)-name such that \(\sym(F(p,\ddx))\geq\sym(\ddx)\) and \(p\forces_\sS\vphi(F(p,\ddx),\ddx)\).

In particular, if \(X\in V\) and \(p\forces_\sS``\ddf\) is a function with domain \(\check{X}\)'' then there are \(\sS\)-names \(\dda_x\) for \(x\in X\) such that \(\sym(\dda_x)\geq\sym(\ddf)\) and \(p\forces_\sS\ddf(\check{x})=\dda_x\).
\end{fact}

\subsubsection{Products}

Given a collection \(\Set{\sS_a=\tup{\bbP_a,\sG_a,\sF_a}\mid a\in A}\) of symmetric systems and an ideal \(\calI\) on \(A\), the \emph{\(\calI\)-support product} of the systems is the symmetric system \(\sS=\tup{\bbP,\sG,\sF}\) given by:
\begin{enumerate}
\item \(\bbP\) is the usual \(\calI\)-support product of \(\Set{\bbP_a \mid a \in A}\). That is, conditions in \(\bbP\) are elements \(p \in \prod_{a \in A}\bbP_a\) satisfying \(\Set{a \in A \mid p(a) \neq \1_{\bbP_a}} \in \calI\).
\item \(\sG\) is the subgroup of the direct product \(\prod_{a \in A}\sG_a\) of elements \(\pi\) satisfying \(\Set{a \in A \mid \pi(a) \neq e_{\sG_a}} \in \calI\).\footnote{Here \(e_{\sG_a}\) is the identity element of the group \(\sG_a\).}
\item \(\sF\) is the set of subgroups of \(\sG\) generated by subgroups of the form \(\prod_{a \in A}H_a\), where \(H_a \in \sF_a\) for all \(a\in A\) and \(\Set{a \in A \mid H_a = \sG_a} \in \calI\). Note that \(\sF\) is in fact a normal filter of subgroups of \(\sG\).
\end{enumerate}

If the symmetric systems are indexed by a set \(\Omega\) of ordinals, then the \emph{Easton-support product} is the product given by the ideal
\begin{equation*}
\calI = \Set{E \subseteq \Omega \mid (\forall \alpha \in \Omega)(\cf(\alpha)=\alpha \implies \abs{E \cap \alpha} < \alpha}.
\end{equation*}
We extend this notion to the case that \(\Omega\) is a class of ordinals similarly.

\subsubsection{Properties of forcings}\label{s:preliminaries;ss:forcing;sss:properties}

Let \(\bbP\) be a forcing. \(\bbP\) is \emph{separative} if for all distinct \(p,p'\in\bbP\) exactly one of the following holds: either \(p' < p\) and \(p \nleq p'\); or there is \(q \leq p'\) such that \(q \perp p\). A set \(X \subseteq \bbP\) is: \emph{open} if for all \(p\in X\) and \(q \leq p\), \(q \in X\); \emph{dense} if for all \(p \in \bbP\) there is \(q \in X\) with \(q \leq p\); and an \emph{antichain} if for all distinct \(p, p' \in X\), \(p \perp p'\). We say that an antichain \(A \subseteq \bbP\) is \emph{maximal} if for all \(p \in \bbP\) there is \(p' \in A\) such that \(p \comp p'\) (equivalently if \(A\) cannot be extended to a larger antichain). For the rest of the paper, we assume that all notions of forcing are separative. This may be done without loss of generality by taking separative quotients.

For an ordinal \(\alpha\) we say that \(\bbP\) is \emph{\(\alpha\)-closed} to mean that any (descending) chain in \(\bbP\) of order type less than \(\alpha\) has a lower bound. Let \(\kappa\) be a cardinal. \(\bbP\) is \emph{\(\kappa\)-distributive} if for all ordinals \(\gamma < \kappa\) and all sets \(\calD = \Set{D_\alpha \mid \alpha < \gamma}\) of open dense subsets of \(\bbP\), \(\bigcap \calD\) open dense. \(\bbP\) has the \emph{\(\kappa\)-chain condition} (also written \(\bbP\) \emph{is \(\kappa\)-c.c.}) if every antichain \(A \subseteq \bbP\) is of cardinality strictly less than \(\kappa\).\footnote{This would be `\(\kappa\text{CC}_2\)' in \cite{karagila_chain_2022}, generalising the notion of \(\text{CCC}_2\) in the expected way, though we shall not make use of the distinctions that \cite{karagila_chain_2022} provides.} \(\bbP\) is \emph{\(\kappa\)-sequential} if for all \(V\)-generic \(G\subseteq\bbP\), \((V^{{<}\kappa})^V=(V^{{<}\kappa})^{V[G]}\).

Let \(\calA\) be a set of maximal antichains of \(\bbP\). A \emph{refinement} of \(\calA\) is a maximal antichain \(A_\ast\) such that, for all \(A\in\calA\), \(p\in A\) and \(p'\in A_\ast\), if \(p \comp p'\) then \(p' \leq p\).

\begin{prop}[\(\ZFC\)]\label{prop:distributive-equivalences}
Let \(\bbP\) be a forcing and \(\kappa\) a cardinal. The following are equivalent:
\begin{enumerate}
\item\label{item:distributive} \(\bbP\) is \(\kappa\)-distributive;
\item\label{item:no-sequences} \(\bbP\) is \(\kappa\)-sequential; and
\item\label{item:refinements} for all \(\gamma<\kappa\) and all sets \(\calA=\Set{A_\alpha\mid\alpha<\gamma}\) of maximal antichains, \(\calA\) has a refinement.
\end{enumerate}
\end{prop}

We have not been able to track down the original proofs that constitute \cref{prop:distributive-equivalences}, but a proof may be found in \cite[Theorem~15.6 and Exercise~15.5]{jech_set_2003}.

\begin{thm}[Folklore]\label{thm:closed-implies-distributive-equivalence}
\(\DC_{{<}\kappa}\) is equivalent to the statement `every \(\kappa\)-closed forcing is \(\kappa\)-distributive.'
\end{thm}

\begin{proof}[Sketch proof]
(\(\impliedby\)).\quad{}Assume that \(\DC_{{<}\kappa}\) fails, witnessed by a \(\gamma\)-closed tree \(T\) with no maximal nodes but no chain of order type \(\gamma\). Then, taking \(T\) to be a forcing by inverting the order operation, we must have that \(T\) is in fact \(\kappa\)-closed. However, \(T\) is not \(\kappa\)-distributive, for the open dense sets \(D_\alpha = \Set{t \in T \mid \rnk_T(t) \geq \alpha}\) have empty intersection.

(\(\implies\)).\quad{}Let \(\bbP\) be a \(\kappa\)-closed forcing and let \(\calD=\Set{D_\alpha\mid\alpha<\gamma}\) be a collection of open dense sets for some \(\gamma<\kappa\). Certainly \(\bigcap\calD\) is open, so it remains to show that it is dense. However, by the \(\kappa\)-closure of \(\bbP\), the density of the \(D_\alpha\) and by \(\DC_\gamma\), we may find a chain \(\tup{p_\alpha \mid \alpha < \gamma}\) such that for all \(\alpha < \beta < \gamma\), \(p \geq p_\alpha \geq p_\beta\) and \(p_\beta \in D_\beta\). Then a lower bound for this chain will be an element of \(\bigcap \calD\).
\end{proof}

The case \(\kappa=\aleph_0\) of \cref{thm:closed-implies-distributive-equivalence} was remarked by Fuchs in \cite{fuchs_geology_2015}.


\begin{defn}
For two notions of forcing \(\bbP\) and \(\bbQ\), a function \(\pi \colon \bbP \to \bbQ\) is a \emph{dense embedding} if:
\begin{enumerate}
\item for all \(p,p' \in \bbP\), \(p' \leq_\bbP p\) if and only if \(\pi(p') \leq_\bbQ \pi(p)\); and
\item \(\pi``\bbP\) is a dense subset of \(\bbQ\).
\end{enumerate}
We say that two notions of forcing are \emph{equivalent} if there is a dense embedding between them.
\end{defn}

We shall sometimes work with complete Boolean algebras, as these are---in a sense---canonical representatives of notions of forcing. Given a complete Boolean algebra, we denote its top element by \(\1\) and its bottom element by \(\0\). When thought of as a notion of forcing, the bottom element is removed from the algebra, but we may still refer to it in some cases for simplicity of notation. Given a forcing \(\bbP\) there is a unique (up to isomorphism) complete Boolean algebra that \(\bbP\) densely embeds into (see \cite[Theorem~14.10]{jech_set_2003}).

\subsection{Small violations of choice}

In \cite{blass_injectivity_1979} Blass introduces the axiom \emph{small violations of choice}. For a set \(S\) (known as the \emph{seed}), \(\SVC(S)\) is the statement `for all \(X\) there is an ordinal \(\eta\) such that \(\abs{X}\leq^\ast\abs{S\times\eta}\),' and \(\SVC^+(S)\) is the injective version `for all \(X\) there is an ordinal \(\eta\) such that \(\abs{X}\leq\abs{S\times\eta}\).' Finally, \(\SVC\) is the statement \((\exists S)\SVC(S)\).

\begin{fact}
For all \(S\), \(\SVC^+(S)\implies\SVC(S)\implies\SVC^+(\power(S))\). Consequently, \((\exists S)\SVC^+(S)\) is equivalent to \(\SVC\).
\end{fact}

Blass demonstrates several key ideas that make models of \(\SVC\) natural to study. For example, he notes that if \(V\models\ZF+\SVC(S)\) then, after forcing with \(S^\lomega\), \(\AC\) will hold. Conversely, if \(\1\forces_\bbP\AC\) then \(\SVC(\bbP)\) holds.

\begin{thm}[Blass]
``A model of \(\ZF\) satisfies \(\SVC\) if and only if some generic extension of it is a model of \(\ZFC\)'' \textup{\cite[Theorem~4.6, p.~45]{blass_injectivity_1979}}.
\end{thm}

Several other notable models of \(\ZF\) are also shown to satisfy \(\SVC\).

\begin{thm}[{Blass \cite[Theorems~4.2 to 4.5]{blass_injectivity_1979}}]
The following are models of \(\SVC\):
\begin{enumerate}
\item permutation submodels of models of \(\ZFA+\AC\) (that is, Fraenkel--Mostowksi models);\footnote{Hall \cite{hall_characterization_2002,hall_permutation_2007} dives deeper into this topic.}
\item symmetric extensions of models of \(\ZFC\);\footnote{This requires considerable machinery from Grigorieff \cite{grigorieff_intermediate_1975} and finding \(\SVC\) seeds for arbitrary symmetric extensions is still a difficult problem.}
\item models of the form \(\HOD(X)\);\footnote{If \(S=(X\cup\Set{X})^\lomega\in\HOD(X)\) then \(\HOD(X)\models\SVC(S)\). Otherwise one requires a certain ``contortion'' to produce the seed.} and
\item models of the form \(L(X)\).\footnote{If \(S\) is the transitive closure of \(X\) then \(L(X)=L(S)\models\SVC(S^\lomega)\).}
\end{enumerate}
\end{thm}

Usuba \cite{usuba_choiceless_2021} later showed that (in \(\ZF\)) being a model of \(\SVC\) is in fact \emph{equivalent} to the assertion that the universe is a symmetric extension of a model of \(\ZFC\). Thus, one concludes \cref{thm:blass-usuba}.

\begin{thm}[{Blass \cite{blass_injectivity_1979} and Usuba \cite{usuba_choiceless_2021}}]\label{thm:blass-usuba}
The following are equivalent:
\begin{enumerate}
\item \(\SVC\).
\item There is a forcing \(\bbP\) such that \(\1 \forces_\bbP \AC\).
\item There is an inner model \(V\subseteq M\) with \(V \models \ZFC\) and a symmetric system \(\sS\in V\) such that, for some \(\sS\)-generic filter \(G\), \(M=V[G]_\sS\).
\item There is an inner model \(V\subseteq M\) with \(V\ \models \ZFC\) and a set \(X\in M\) such that \(M=V(X)\).
\end{enumerate}
\end{thm}

\subsection{Inner models of \(\ZF\)}

We shall require the following theorem that first appeared in \cite{hajnal_consistency_1956} (but can also be found as \cite[Theorem~13.9]{jech_set_2003}) to show that the symmetric extension in \cref{thm:new-models} is a model of \(\ZF\).

\begin{thm}[Hajnal]\label{thm:inner-model}
A transitive class \(M\) is an inner model of \(\ZF\) if and only if it is \emph{almost universal} (that is, for all sets \(x\subseteq M\) there is \(y\in M\) such that \(x\subseteq y\)) and closed under the following \emph{G\"odel operations:}
\begin{align*}
G_1(X,Y)&=\Set{X,Y}&G_6(X)&=\bigcup X\\
G_2(X,Y)&=X\times Y&G_7(X)&=\Set{u\mid(\exists v)\tup{u,v}\in X}\\
G_3(X,Y)&=\Set{\tup{u,v}\in X\times Y\mid u\in v}&G_8(X)&=\Set{\tup{u,v}\mid\tup{v,u}\in X}\\
G_4(X,Y)&=X\setminus Y&G_9(X)&=\Set{\tup{u,v,w}\mid\tup{u,w,v}\in X}\\
G_5(X,Y)&=X\cap Y&G_{10}(X)&=\Set{\tup{u,v,w}\mid\tup{v,w,u}\in X}.
\end{align*}
\end{thm}

\section{Descending distribution}\label{s:generosity}

\begin{defn}\label{defn:generous}
A forcing \(\bbP\) is \emph{\(\kappa\)-descending distributive} if for all decreasing sequences \(\tup{D_\alpha \mid \alpha < \kappa}\) of open dense sets (that is, for \(\alpha < \beta < \kappa\), \(D_\alpha \supseteq D_\beta\)), \(\bigcap_{\alpha < \kappa} D_\alpha\) is open dense.

We also have the following equivalent characterisation (see \cref{thm:generous-equivalence-decreasing-open-dense}). Let \(\calD = \tup{ D_\alpha \mid \alpha < \kappa }\) be an arbitrary sequence of open dense subsets of a forcing \(\bbP\), where \(\kappa\) is a cardinal. We associate to \(\calD\) the rank function \(\rnk_\calD \colon \bbP \to \kappa + 1\) given by \(\rnk_\calD(p) = \sup \Set{ \alpha < \kappa \mid p \in D_\alpha}\). Then a forcing \(\bbP\) is \(\kappa\)-descending distributive if for all \(\calD\), \(\Set{p \in \bbP \mid \rnk_\calD(p) = \kappa}\) is open dense.

\(\bbP\) is \emph{strongly \(\kappa\)-descending distributive} if for all sequences \(\tup{D_\alpha \mid \alpha < \kappa}\) of open dense sets there is unbounded \(I \subseteq \kappa\) such that \(\bigcap_{\alpha \in I}D_\alpha\) is open dense.
\end{defn}

\begin{prop}\label{thm:generous-boolean-completion}
If \(\pi \colon \bbP \to \bbQ\) is a dense embedding and \(\bbQ\) is \(\kappa\)-descending distributive then \(\bbP\) is also \(\kappa\)-descending distributive. In particular, a forcing \(\bbP\) is \(\kappa\)-descending distributive if and only if its Boolean completion is.
\end{prop}

\begin{proof}
Let \(\calD = \tup{D_\alpha \mid \alpha < \kappa}\) be a decreasing sequence of open dense subsets of \(\bbP\). Since \(\bigcap \calD\) is automatically open, it remains to show that it is dense, so let \(p_0 \in \bbP\). For \(\alpha < \kappa\), let \(E_\alpha\) be the open dense set \(\Set{q \in \bbQ \mid (\exists p \in D_\alpha) q \leq \pi(p)}\). Note that \(\tup{E_\alpha \mid \alpha < \kappa}\) is a decreasing sequence of open dense subsets of \(\bbQ\). By \(\kappa\)-descending distributivity, \(\bigcap_{\alpha < \kappa} E_\alpha\) is open dense so, in particular, there is \(q \leq \pi(p_0)\) with \(q \in \bigcap_{\alpha < \kappa} E_\alpha\). Let \(p\) be such that \(\pi(p) \leq q\). Then for all \(\alpha < \kappa\) there is \(p^{(\alpha)} \in D_\alpha\) such that \(\pi(p) \leq q \leq \pi(p^{(\alpha)})\), so \(p \leq p^{(\alpha)}\) and in fact \(p \in \bigcap \calD\). Since \(\pi(p) \leq q \leq \pi(p_0)\), we also have \(p \leq p_0\) as required.
\end{proof}

A similar proof shows that a forcing is strongly \(\kappa\)-descending distributive if and only if its Boolean completion is. While we do not spend much time on strong descending distributivity, the following question comes to mind.

\begin{qn}\label{qn:generous-is-strong}
Is every \(\kappa\)-descending distributive forcing also strongly \(\kappa\)-descending distributive?
\end{qn}

Since a finite product of strongly \(\kappa\)-descending distributive forcings is strongly \(\kappa\)-descending distributive, so showing that \(\kappa\)-descending distributivity is \emph{not} in general preserved by finite products would answer \cref{qn:generous-is-strong} negatively.

\begin{prop}\label{cor:generosity-is-cofinal}
\(\bbP\) is \(\kappa\)-descending distributive if and only if it is \(\cf(\kappa)\)-descending distributive.
\end{prop}

\begin{proof}
Let \(\tup{\alpha_\gamma \mid \gamma < \cf(\kappa)}\) be a strictly increasing cofinal sequence in \(\kappa\).

Assume first that \(\bbP\) is \(\kappa\)-descending distributive and let \(\calD = \tup{D_\gamma \mid \gamma < \cf(\kappa)}\) be a decreasing sequence of open dense sets. Define \(D_\alpha' = D_\gamma\), where \(\gamma\) is least such that \(\alpha_\gamma \geq \alpha\). Then \(\calD' = \tup{D_\alpha' \mid \alpha < \kappa}\) is a decreasing sequence of open dense sets and so, by assumption, \(\bigcap \calD' = \bigcap \calD\) is open dense as required.

Assume instead that \(\bbP\) is \(\cf(\kappa)\)-descending distributive and let \(\calD = \tup{D_\alpha \mid \alpha < \kappa}\) be a decreasing collection of open dense sets. By \(\cf(\kappa)\)-descending distributivity, \(\bigcap_{\gamma < \cf(\kappa)} D_{\alpha_\gamma} = \bigcap \calD\) is open dense as required.
\end{proof}

The following are very common ways in which descending distributivity can manifest and, ultimately, inspired the original definition of descending distributivity.\footnote{For example, \cref{thm:al-preservation} used to be two theorems, one of which dealt with situations in which \(\kappa \geq \aleph^\ast(\bbP)\) and the other dealt with situations in which \(\bbP\) is \(\kappa^+\)-distributive (plus analogues for the filter of subgroups).}

\begin{prop}\label{prop:common-generosity}
If \(\cf(\kappa) \geq \aleph^\ast(\bbP)\) or if \(\bbP\) is \(\kappa^+\)-distributive then \(\bbP\) is \(\kappa\)-descending distributive.
\end{prop}

\begin{proof}
Assume first that \(\cf(\kappa) \geq \aleph^\ast(\bbP)\). By \cref{cor:generosity-is-cofinal} we may assume that \(\kappa\) is regular. Let \(\calD\) be a decreasing sequence of \(\kappa\)-many open dense subsets and let \(p_0 \in \bbP\). We shall find \(p \in \bigcap \calD\) with \(p \leq p_0\). Let \(\alpha_p = \min\Set{\alpha \leq \kappa \mid p \notin D_\alpha}\), defining \(D_\kappa = \emptyset\). Then \(\sup\Set{\alpha_p \mid p \leq p_0} = \kappa\), \(\kappa\) is regular and \(\kappa \geq \aleph^\ast(\bbP)\), so there is \(p \leq p_0\) with \(\alpha_p = \kappa\). That is, \(p \in \bigcap \calD\) as required.

If \(\bbP\) is \(\kappa^+\)-distributive then we automatically have that \(\bigcap \calD\) is open dense.
\end{proof}

In particular, if \(\bbP\) is \(\kappa^+\)-closed and \(\DC_\kappa\) holds for subtrees of \(\bbP^{{<}\kappa}\) then \(\bbP\) is \(\kappa^+\)-distributive and thus \(\kappa\)-descending distributive.

\cref{prop:generous-doesnt-change-cf} provides an early intuition about how \(\kappa\)-descending distributive notions of forcing behave, though it should be noted that really it is just an immediate corollary of \cref{thm:generous-implies-no-fresh}.

\begin{prop}\label{prop:generous-doesnt-change-cf}
If \(\bbP\) is \(\kappa\)-descending distributive then \(\bbP\) does not change the cofinality of \(\kappa\).
\end{prop}

\begin{proof}
Let \(\ddf\) be a \(\bbP\)-name for an cofinal increasing injection \(\lambda \to \kappa\) for some ordinal \(\lambda\). For \(\alpha < \kappa\), let \(D_\alpha = \Set{ p \mid (\exists \beta > \alpha) (\exists \eta < \lambda) p \forces_\bbP \ddf(\check{\eta}) = \check{\beta}}\). Since \(\ddf\) is forced to be cofinal, \(D_\alpha\) is open dense and \(\calD = \tup{D_\alpha \mid \alpha < \kappa}\) forms a decreasing sequence. By descending distributivity, \(\bigcap \calD\) is open dense. However, if \(p \in \bigcap \calD\) then \(p\) decides the value of \(\ddf(\check{\eta})\) for cofinally many \(\eta < \lambda\) and we obtain a cofinal sequence \(\cf(\lambda) \to \kappa\) in the ground model. Thus \(\cf(\kappa) \leq \cf(\lambda) \leq \lambda\).
\end{proof}

Let us also mention that chain conditions have nothing to say about descending distributivity. Recall that, for cardinals \(\eta\) and \(\lambda\), \(\Add(\eta,\lambda)\) is the notion of forcing with conditions that are partial functions \(p \colon \lambda \times \eta \to 2\) such that \(\abs{\dom(p)} < \eta\), with \(q \leq p\) if \(q \supseteq p\). In \(\ZFC\), if \(\eta\) is regular infinite and \(\lambda > 0\) then \(\Add(\eta,\lambda)\) has the \((\eta^{{<}\eta})^+\)-chain condition.

\begin{prop}[\(\ZFC\)]\label{prop:add-generous-spectrum}
Let \(\eta \geq \aleph_0\), \(\lambda > 0\) and \(\kappa=\cf(\kappa) \geq \aleph_0\). Then \(\Add(\eta,\lambda)\) is \(\kappa\)-descending distributive if \(\kappa < \cf(\eta)\) or \(\kappa > (\eta\times\lambda)^{{<}\eta}\). \(\Add(\eta,\lambda)\) is not \(\kappa\)-descending distributive if \(\cf(\eta) \leq \kappa \leq \eta^{{<}\eta} \times \lambda\).
\end{prop}

The gaps left here are those regular \(\kappa\) such that \(\eta^{{<}\eta} \times \lambda < \kappa \leq \lambda^{{<}\eta}\). For example, we do not know if \(\Add(\omega_1,\omega_\omega)\) is \(\omega_{\omega+1}\)-descending distributive under \(\GCH\). Note also that it is sufficient to consider only regular \(\kappa\) by \cref{cor:generosity-is-cofinal}.

\begin{proof}
Let \(\bbP = \Add(\eta,\lambda)\). Since \(\Add(\eta,1)\) is equivalent to \(\Add(\eta,\eta)\), we may replace \(\lambda\) by \(\lambda + \eta\) and assume for the rest of the proof that \(\lambda \geq \eta\).

If \(\kappa < \cf(\eta)\) then \(\bbP\) is \(\kappa\)-descending distributive because it is \(\cf(\eta)\)-distributive. If instead \(\kappa > \lambda^{{<}\eta}\) then \(\bbP\) is \(\kappa\)-descending distributive because \(\kappa > \abs{\bbP}\), meaning \(\kappa \geq \aleph^\ast(\bbP)\).

Suppose that \(\cf(\eta) < \kappa \leq \eta\). Since \(\bbP\) collapses \(\eta\) to \(\cf(\eta)\), we have changed the cofinality of \(\kappa\), so \(\bbP\) cannot be \(\kappa\)-descending distributive by \cref{prop:generous-doesnt-change-cf}.

For the case \(\kappa = \cf(\eta)\), we shall show that \(\bbP\) is not \(\eta\)-descending distributive and appeal to \cref{cor:generosity-is-cofinal}. For \(p \in \bbP\), let \(s(p) = \Set{\alpha < \eta \mid (\exists \beta < \lambda) \tup{\beta,\alpha} \in \dom(p)}\). Let \(D_\alpha = \Set{p \in \bbP \mid \alpha \subseteq s(p)}\), noting that \(\tup{D_\alpha \mid \alpha < \eta}\) forms a decreasing sequence of open dense sets. However, if \(p \in \bigcap_{\alpha < \eta} D_\alpha\) then \(\eta = \abs{s(p)} \leq \abs{\dom(p)}\), so such a \(p\) cannot exist. Thus \(\bigcap_{\alpha < \eta} D_\alpha = \emptyset\) as required.

Finally, if \(\eta \leq \kappa \leq \lambda\) then let \(D_\alpha = \Set{ p \in \bbP \mid \dom(p) \cap ((\kappa \setminus \alpha) \times \eta) \neq \emptyset}\), noting that \(\tup{D_\alpha \mid \alpha < \kappa}\) forms a decreasing sequence of open dense sets. If \(p \in \bigcap_{\alpha < \kappa} D_\alpha\) then there is cofinal \(I \subseteq \kappa\) such that, for all \(\alpha \in I\), \(\dom(p) \cap \Set{\alpha} \times \eta \neq \emptyset\). Thus, \(\abs{\dom(p)} \geq \abs{I} \geq \cf(\kappa) = \kappa \geq \eta\), so no such \(p\) exists and \(\bigcap_{\alpha < \kappa} D_\alpha = \emptyset\).
\end{proof}

\begin{cor}[\(\ZFC\)]\label{cor:add-generous-spectrum-exact}
If \(\lambda^{{<}\eta}=\lambda \geq \aleph_0\) then \(\Add(\eta,\lambda)\) is \(\kappa\)-descending distributive if and only if \(\cf(\kappa) < \cf(\eta)\) or \(\cf(\kappa) > \lambda^{{<}\eta}\).
\end{cor}

Thus, for example, \(\Add(\omega,\lambda)\) is c.c.c.\ but not \(\kappa\)-descending distributive for any \(\cf(\kappa) \leq \lambda\).

\subsection{Equivalent characterisations}

As mentioned in \cref{defn:generous}, we may characterise descending distributivity through associating rank functions to arbitrary sequences of open dense sets.

\begin{thm}\label{thm:generous-equivalence-decreasing-open-dense}
\(\bbP\) is \(\kappa\)-descending distributive if and only if for all \(\kappa\)-sequences of open dense sets, \(\Set{p \in \bbP \mid \rnk_\calD(p) = \kappa}\) is open dense.
\end{thm}

\begin{proof}
First assume that \(\bbP\) is \(\kappa\)-descending distributive and let \(\calD = \tup{D_\alpha \mid \alpha < \kappa}\) be an arbitrary collection of dense open sets. Define \(\calE = \tup{E_\alpha \mid \alpha < \kappa}\) by setting \(E_\alpha = \bigcup_{\beta > \alpha} D_\beta\). Then \(\calE\) is a decreasing sequence of open dense sets, so \(\bigcap \calE\) is open dense. However, \(p \in \bigcap \calE\) if and only if for all \(\alpha < \kappa\) there is \(\beta > \alpha\) such that \(p \in D_\beta\), which is precisely to say that \(\rnk_\calD(p)=\kappa\).

On the other hand, if \(\calD=\tup{D_\alpha \mid \alpha < \kappa}\) is decreasing, then \(\rnk_\calD(p) = \kappa\) if and only if \(p \in \bigcap \calD\), so \(\bigcap \calD\) is indeed open dense as required.
\end{proof}

The following characterisation of descending distributivity is inspired by the distributivity laws that give rise to the name `distributive' in notions of forcing (see \cite[Lemma~7.16]{jech_set_2003}, the proof of which we follow closely for \cref{prop:generous-boolean-algebra-characterisation}).

\begin{prop}\label{prop:generous-boolean-algebra-characterisation}
For a complete Boolean algebra \(\bbB\), the following are equivalent:
\begin{enumerate}
\item The \emph{\(\kappa\)-descending distributivity law} stating that, for all decreasing sequences \(\tup{B_\alpha \mid \alpha < \kappa}\) of non-empty subsets of \(\bbB\) we have
\begin{equation}\label{eqn:downwards-distributivity-law}
\prod_{\alpha < \kappa} \sum B_\alpha = \sum_{f \in \prod_{\alpha < \kappa}^{\catSet} B_\alpha} \prod f``\kappa.\tag{\ensuremath{\ast}}
\end{equation}
Here \(\prod_{\alpha < \kappa}^{\catSet} B_\alpha\) is the \emph{set} of \emph{choice functions} for the family \(\tup{B_\alpha \mid \alpha < \kappa}\).
\item \(\bbB\) is \(\kappa\)-descending distributive.
\end{enumerate}
\end{prop}

\begin{proof}
Assume first that \(\bbB\) is \(\kappa\)-descending distributive and let \(\tup{B_\alpha \mid \alpha < \kappa}\) be a decreasing sequence of non-empty subsets of \(\bbB\). Let us first prove that the right hand side of \cref{eqn:downwards-distributivity-law} is always less than or equal to the left hand side. For \(f \in \prod_{\alpha < \kappa}^{\catSet} B_\alpha\), let \(u_f = \prod f``\kappa\). Then we must have that \(u_f \leq f(\alpha)\) for all \(\alpha\) and thus \(u_f \leq \sum B_\alpha\). Hence \(\sum_f u_f \leq \sum B_\alpha\) and so
\begin{equation*}
\sum_f \prod f``\kappa = \sum_f u_f \leq \prod_\alpha \sum B_a.
\end{equation*}
Thus it remains to show that the left hand side of \cref{eqn:downwards-distributivity-law} is less than or equal to the right hand side. Let \(u = \prod_\alpha \sum B_\alpha\) and, for all \(\alpha\), define the open dense set \(D_\alpha = \Set{b \in \bbB \mid (\exists a \in B_\alpha) b \leq a \land u} \cup \Set{ b \in \bbB \mid b \land u = \0}\). Note that necessarily this also forms a decreasing sequence of open dense sets. By \(\kappa\)-descending distributivity, \(D = \bigcap_{\alpha < \kappa} D_\alpha\) is open dense. Suppose that \(v \leq u\) and \(v \in D\). Then this must be witnessed by \(f \in \prod_{\alpha < \kappa}^{\catSet} B_\alpha\), where \(v \leq f(\alpha) \land u\) for all \(\alpha < \kappa\). Since \(D\) is dense below \(u\), it must be that \(\sum_f u_f \geq u\) as required.

On the other hand, assume that \(\bbB\) satisfies the \(\kappa\)-descending distributivity law. Let \(\tup{D_\alpha \mid \alpha < \kappa}\) be a decreasing sequence of open dense sets and let \(u \in \bbB\). Note that for all \(\alpha < \kappa\), \(\sum D_\alpha = \1\) and so \(\prod_\alpha \sum D_\alpha = \1\). Hence \(\sum_f \prod f``\kappa = \1\), so in particular there is \(f\) such that \(u_f \land u \neq \0\). However, since the \(D_\alpha\) are open, we may define \(g(\alpha) = f(\alpha) \land u \in D_\alpha\). Then \(u_g = u_f \land u \leq u\). Furthermore, for all \(\alpha < \kappa\), \(u_g \leq g(\alpha) \leq f(\alpha) \in D_\alpha\), so \(u_g \in D_\alpha\). That is, \(u_g \in \bigcap_{\alpha < \kappa} D_\alpha\) with \(u_g \leq u\) as required.
\end{proof}

Analogous to `\ref{item:distributive} if and only if \ref{item:refinements}' in \cref{prop:distributive-equivalences}, if we assume \(\AC\) then we can rephrase \(\kappa\)-descending distributivity in terms of maximal antichains. We omit the proof as it is extremely similar.

\begin{prop}[\(\ZFC\)]\label{prop:generous-equivalent-antichain-zfc}
A forcing \(\bbP\) is \(\kappa\)-descending distributive if and only if for all sequences \(\calA = \tup{A_\alpha \mid \alpha < \kappa}\) of maximal antichains there are densely many \(q\) satisfying that there is an unbounded set \(I \subseteq \kappa\) such that, for all \(\alpha \in I\) and \(p \in A_\alpha\), if \(p \comp q\) then \(q \leq p\).

A forcing \(\bbP\) is strongly \(\kappa\)-descending distributive if and only if for all such \(\calA\) there is unbounded \(I \subseteq \kappa\) such that \(\Set{A_\alpha \mid \alpha \in I}\) has a refinement.
\end{prop}

\begin{prop}
Let \(\bbP\) be a forcing. Say that \(\bbP\) is `\({<}\kappa\)-descending distributive' if for all \(\lambda < \kappa\), \(\bbP\) is \(\lambda\)-descending distributive. Then \(\bbP\) is \({<}\kappa\)-descending distributive if and only if it is \(\kappa\)-distributive. In particular, \(\bbP\) is \(\omega\)-descending distributive if and only if it is \(\sigma\)-distributive (that is, \(\omega_1\)-distributive).
\end{prop}

\begin{proof}
We go by induction on infinite \(\kappa\). Since finite intersections of open dense sets are open dense, every forcing is automatically both \({<}\omega\)-descending distributive and \(\omega\)-distributive.

Suppose that \(\bbP\) is \({<}\kappa\)-descending distributive. Then tautologically, for all \(\lambda < \kappa\), \(\bbP\) is \({<}\lambda\)-descending distributive. Thus, by induction, \(\bbP\) is \(\lambda\)-distributive for all \(\lambda < \kappa\). If \(\kappa\) is a limit then we conclude that \(\bbP\) is \(\kappa\)-distributive. If instead \(\kappa = \lambda^+\), let \(\calD = \tup{D_\alpha \mid \alpha < \lambda}\). By \(\lambda\)-distributivity, we may replace \(D_\alpha\) by \(\bigcap_{\beta \leq \alpha} D_\beta\) to assume without loss of generality that \(\calD\) is a decreasing sequence. By \(\lambda\)-descending distributivity, \(\bigcap \calD\) is open dense as required. Thus, \(\bbP\) is \(\kappa\)-distributive.

On the other hand, if \(\bbP\) is \(\kappa\)-distributive then we automatically have that it is \(\lambda\)-descending distributive for all \(\lambda < \kappa\) as required.
\end{proof}

\subsection{Fresh functions}

The following definition is due to Hamkins \cite{hamkins_gap_2001}. Note that Hamkins---and many later works, such as \cite{fischer_fresh_2023}---make the additional assumption that a fresh function has ordinal image, which they may do without risk as they are working in \(\ZFC\). We shall not do this, as we are working in \(\ZF\).

\begin{defn}
Let \(\kappa\) be a cardinal and \(V\) a model of \(\ZF\). A function \(f \colon \kappa \to V\) is \emph{fresh function on \(\kappa\) over \(V\)} if \(f \notin V\) but, for all \(\alpha < \kappa\), \(f \res \alpha \in V\). We may omit `on \(\kappa\)' or `over \(V\)' if clear from context.
\end{defn}

It turns out that a forcing being \(\kappa\)-descending distributive implies that it adds no fresh functions on \(\kappa\) (see \cref{thm:generous-implies-no-fresh}). We also have a partial converse, but this requires a strange and seemingly quite strong assumption.

\begin{defn}
A forcing \(\bbP\) is \emph{chain-like} if for all \(V\)-generic \(G \subseteq \bbP\) and all open dense \(D \subseteq \bbP\) with \(D\in V\), \(G \setminus D \in V\).
\end{defn}

Note that if any generic filter is a chain then one automatically obtains aspirational linearity. In particular, if a forcing is a tree then this occurs, though not every forcing can be presented as a tree (see \cite{konig_dense_2006}). Furthermore, like being a tree, aspirational linearity is not a property of the forcing, but of the \emph{presentation}. That is, as we will show upon the conclusion of this sentence, there are equivalent notions of forcing \(\bbP\) and \(\bbQ\) such that \(\bbP\) is chain-like, but \(\bbQ\) is not.

\begin{prop}
Let \(\bbP\) be the forcing with conditions that are functions \(n \to 2\) for \(n < \omega\), ordered by \(q \leq p\) if \(q \supseteq p\). Then \(\bbP\) is chain-like.

Let \(\bbQ\) be the forcing with conditions that are \emph{arbitrary finite partial functions} \(\omega \to 2\), ordered by \(q \leq p\) if \(q \supseteq p\). Then \(\bbQ\) is not chain-like.
\end{prop}

\begin{proof}
Let \(G\) be \(V\)-generic for \(\bbP\) and let \(D \in V\) be an open dense subset of \(\bbP\). Let \(p \in G \cap D\) be of minimal domain. Then \(G \setminus D = \Set{p \res n \mid n < \dom(p)} \in V\).

Let \(D\) be the set of conditions \(p\) such that \(\dom(p)\) contains an even number. If \(H \subseteq \bbQ\) is \(V\)-generic then \(H\setminus D\) encodes the \(V\)-generic Cohen real given by the odd co-ordinates of \(H\) and thus is, itself, Cohen generic.
\end{proof}

In this case, the identity function \(\bbP \to \bbQ\) is a dense embedding.

\begin{thm}\label{thm:generous-implies-no-fresh}
If \(\bbP\) is \(\kappa\)-descending distributive then it adds no fresh functions on \(\kappa\).

If \(\bbP\) adds no fresh functions on \(\kappa\) and is chain-like then it is \(\kappa\)-descending distributive.
\end{thm}

\begin{proof}
Let \(\ddf\) be a name for a function \(\kappa \to V\) such that \(p \forces_\bbP(\forall \alpha < \check{\kappa})\ddf \res \alpha \in \check{V}\). Let \(D_\alpha\) be the open dense set of all \(q\) incompatible with \(p\) and all \(q \leq p\) deciding \(g \in V\) such that \(\ddf\res\check{\alpha} = \check{g}\). Note that if \(\beta < \alpha\) then \(q \forces_\bbP \ddf\res\check{\beta} = \check{g}\res\check{\beta}\), so the sequence \(D_\alpha\) is decreasing. Let \(q \leq p\) be an element of \(\bigcap_{\alpha < \kappa} D_\alpha\). Then \(q\) decides \(\ddf(\check{\alpha})\) for all \(\alpha\), so \(q \forces_\bbP \ddf \in \check{V}\) and \(\ddf\) is not fresh.

On the other hand, assume that \(\bbP\) is chain-like and adds no fresh functions \(\kappa \to V\). Let \(\calD = \tup{D_\alpha \mid \alpha < \kappa}\) be a descending sequence of open dense sets and \(p \in \bbP\). If there is an atom \(q \leq p\) then certainly \(q \in \bigcap \calD\),\footnote{This uses our assumption of separativity, for this may not be true if \(\bbP\) is not separative. Consider the partial order \(\omega^\ast\), that is the order type of \(\Set{x \in \bbZ \mid x < 0}\). Then \(\bbP\) is not \(\omega\)-descending distributive, witnessed by the sequence of dense open sets \(D_n = \Set{x \in \bbZ \mid x < n}\).} so assume instead that there are no atoms below \(p\). Let \(G\) be \(V\)-generic for \(\bbP\) with \(p \in G\). Define the name \(\ddf\) so that \(\ddf^G(\alpha) = G \setminus D_\alpha\) for all generic filters \(G\). Since \(\bbP\) is chain-like, \(\ddf^G\res\alpha \in V\) for all \(\alpha < \kappa\) and, since \(\bbP\) adds no fresh functions, \(\ddf^G \in V\). Hence \(G_0 = G \setminus \bigcap \calD \in V\) as well. However, \(G \notin V\) (by the lack of atoms below \(p\)), so it must be that \(G_0 \neq G\). Therefore there is \(q \in G\) such that \(q \in \bigcap \calD\). In particular, \(q \comp p\) so, letting \(r \leq p,q\), we have that \(r \in \bigcap \calD\) with \(r \leq p\) as required.
\end{proof}

Using \cref{thm:generous-implies-no-fresh} and the following proposition of Fischer, Koelbing and Wohofsky, we may show that not every notion of forcing is equivalent to a chain-like one and that \(\kappa\)-descending distributivity is in fact a stronger condition than not adding fresh functions on \(\kappa\). Here \(\FRESH(\bbP)\) (the \emph{fresh function spectrum} of \(\bbP\), see \cite{fischer_fresh_2023}) is the set of all \emph{regular} cardinals \(\delta\) such that \(\bbP\) can add a fresh function on \(\delta\).

\begin{prop}[Fischer--Koelbing--Wohofsky]\label{prop:fischer-fresh}
``If \(\bbP\) has the \(\chi\)-c.c.\ and \(\delta > \chi\) then \(\delta \notin \FRESH(\bbP)\)'' \textup{\cite[Proposition~2.5, p.~4]{fischer_fresh_2023}.}
\end{prop}

\begin{cor}
If \(\lambda\) and \(\eta\) are regular infinite cardinals with \(\lambda > (\eta^{{<}\eta})^+\) then, for all regular \(\kappa\) with \((\eta^{{<}\eta})^+ < \kappa \leq \lambda\), \(\Add(\eta,\lambda)\) adds no fresh functions on \(\kappa\) but is not \(\kappa\)-descending distributive.

In particular, \(\Add(\eta,\lambda)\) is not equivalent to a chain-like notion of forcing and \(\kappa\)-descending distributivity is not equivalent to not adding fresh functions on \(\kappa\).
\end{cor}

\begin{proof}
By \cref{prop:fischer-fresh} (and that \(\Add(\eta,\lambda)\) is \((\eta^{{<}\eta})^+\)-c.c.), \(\Add(\eta,\lambda)\) adds no fresh functions on \(\kappa\). However, by \cref{prop:add-generous-spectrum}, \(\Add(\eta,\lambda)\) is not \(\kappa\)-descending distributive. If \(\bbP\) were equivalent to a chain-like forcing, then \cref{thm:generous-implies-no-fresh} would give us that it is \(\kappa\)-descending distributive, so this cannot be the case.
\end{proof}

\subsection{Sequentiality}

While sequentiality and distributivity are identical in \(\ZFC\), this need not be the case without choice. The following result of Karagila and Schilhan is a stark demonstration of this.

\begin{thm}[Karagila--Schilhan]\label{thm:karagila-schilhan-dc}
``Let \(\kappa\) be any infinite cardinal. It is consistent with \(\ZF+\DC_{{<}\kappa}\) that:
\begin{enumerate}
\item\label{item:karagila-dc} There is a \(\kappa\)-distributive forcing which violates \(\DC\).
\item\label{item:karagila-ac-omega} There is a \(\kappa\)-sequential forcing which violates \(\AC_\omega\).''
\end{enumerate}
\textup{\cite[Theorem~5.1, p.~875]{karagila_sequential_2023}}
\end{thm}

In this case, Karagila and Schilhan work in a model of \(\ZF+\DC_{{<}\kappa}\) in which there is a non-well-orderable set \(A\) such that every well-orderable subset of \(A\) has cardinality less than \(\kappa\). The forcing \(\bbQ_1\) that witnesses \cref{item:karagila-ac-omega} has conditions given by finite partitions of well-orderable subsets of \(A\). That is, a condition \(e\) is a finite set of pairwise disjoint well-orderable subsets of \(A\). The ordering is given by \(f \leq_{\bbQ_1} e\) if \(\Set{C \cap \bigcup e \mid C \in f} = e\); intuitively, `cells' \(C \in e\) may be extended and new cells may be added, but no cells may be merged.

\begin{prop}[Karagila--Schilhan]\label{prop:q1-not-distributive}
\(\bbQ_1\) is not \(\omega\)-distributive.
\end{prop}

\begin{proof}
Let \(D_n\) be the set of conditions \(e\) such that \(\abs{e} \geq n\). Since for all \(e \in \bbQ\), \(A \setminus \bigcup e\) is infinite, we may always extend \(e\) to \(f \in D_n\) be adding additional singletons from \(A \setminus \bigcup e\). Thus, \(\tup{D_n \mid n < \omega}\) is a decreasing sequence of open dense subsets of \(\bbQ_1\). However, \(\bigcap_{n < \omega} D_n = \emptyset\).
\end{proof}

\begin{cor}
It is consistent with \(\ZF+\DC_{{<}\kappa}\) that there is a \(\kappa\)-sequential forcing which is not \(\omega\)-descending distributive.
\end{cor}

On the other hand, \cref{prop:common-generosity,thm:generous-implies-no-fresh} are \(\ZF\) results, so the forcing witnessing \cref{item:karagila-dc} of \cref{thm:karagila-schilhan-dc} will not add fresh functions on any \(\alpha < \kappa\).

\subsection{Filters of subgroups}

Let \(\sF\) be a filter on the partial order \(\tup{A,{\leq}}\). Then \(\sF\) inherits the partial order from \(A\), so we may translate many notions from forcing into this language. However, by the definition of a filter, \(\sF\) has a special property that usually relegates a notion of forcing to uselessness: any two elements of \(\sF\) are compatible. Therefore, any open set is dense and, in particular, the principal open sets \([a]=\Set{b \in \sF \mid b \leq a}\) are open dense. Let \(\kappa\) be a cardinal and assume that all \(\kappa\)-sequences of open dense subsets of \(\sF\) have a selector (for example, \(\sF\) has a well-orderable dense base, or \(\AC_\kappa\) holds). Then \(\kappa\)-sequences of open dense sets may always be shaved down to \(\kappa\)-sequences of principal open dense sets, casting notions such as distributivity in a new light: \(\sF\) is \(\kappa\)-complete if and only if it is \(\kappa\)-distributive if and only if it is \(\kappa\)-closed. Thinking of \(\sF\) as a (normal) filter of subgroups in a symmetric system, one can re-phrase the folklore result that `if \(\bbP\) is \(\kappa\)-closed and \(\sF\) is \(\kappa\)-complete then \(\tup{\bbP,\sG,\sF}\) preserves \(\DC_{{<}\kappa}\)' as `if \(\bbP\) and \(\sF\) are \(\kappa\)-closed then \(\tup{\bbP,\sG,\sF}\) preserves \(\DC_{{<}\kappa}\)'.\footnote{A proof can be found in, say, \cite[Lemma~2.1]{karagila_embedding_2014}.}

In the settings that we wish to use the descending distributivity of a filter, we will always be able to pick out elements of the appropriate open dense sets. This is because such open dense sets always manifest precisely as \([\sym(\ddx)]\) for some \(\sS\)-name \(\ddx\). Hence, we provide the following stronger definition of descending distributivity in the case of filters.

\begin{defn}
A filter \(\sF\) on a partial order \(\tup{A,{\leq}}\) is \emph{\(\kappa\)-descending distributive} if for all sequences \(\tup{a_\alpha \mid \alpha < \kappa}\) of elements of \(\sF\) there is \(b \in \sF\) such that \(\Set{\alpha < \kappa \mid b \leq a_\alpha}\) is unbounded in \(\kappa\).\footnote{Note that, in this case, \(\kappa\)-descending distributivity and strong \(\kappa\)-descending distributivity are genuinely identical.}

In particular, if \(\sF\) is a filter of subgroups of \(\sG\), \(\sF\) is \(\kappa\)-descending distributive if for all sequences \(\tup{H_\alpha \mid \alpha < \kappa}\) from \(\sF\) there is unbounded \(I \subseteq \kappa\) such that \(\bigcap_{\alpha \in I}H_\alpha \in \sF\).

We say that a symmetric system \(\tup{\bbP,\sG,\sF}\) is \(\kappa\)-descending distributive if both \(\bbP\) and \(\sF\) are.
\end{defn}

One disadvantage to this definition is that we lose an equivalent of \cref{thm:generous-equivalence-decreasing-open-dense}, since the group generated by \(H \cup H'\) will usually have subgroups in \(\sF\) that are subgroups of neither \(H\) nor \(H'\). In particular, the name becomes somewhat less sensible. However, we do retain an analogue of \cref{cor:generosity-is-cofinal}.

\begin{prop}\label{prop:generosity-is-cofinal-filters}
A filter \(\sF\) is \(\kappa\)-descending distributive if and only if it is \(\cf(\kappa)\)-descending distributive.
\end{prop}

\begin{proof}
Firstly, let \(\tup{\alpha_\gamma \mid \gamma < \cf(\kappa)}\) be a strictly increasing cofinal sequence in \(\kappa\).

Assume that \(\sF\) is \(\kappa\)-descending distributive and let \(\tup{a_\gamma \mid \gamma < \cf(\kappa)}\) be a sequence of elements from \(\sF\). Let \(a_\alpha' = a_{\alpha_\gamma}\), where \(\gamma\) is least such that \(\alpha_\gamma \geq \alpha\). Then, by \(\kappa\)-descending distributivity, there is \(a \in \sF\) lying below cofinally many \(a_\alpha'\). However, this implies that \(a\) is below cofinally many \(a_\gamma\) as well.

On the other hand, if \(\sF\) is \(\cf(\kappa)\)-descending distributive and \(\tup{a_\alpha \mid \alpha < \kappa}\) is a sequence of elements from \(\sF\) then, by \(\cf(\kappa)\)-descending distributivity, there is \(a \in \sF\) lying below cofinally many of the elements of the sequence \(\tup{a_{\alpha_\gamma} \mid \gamma < \cf(\kappa)}\). In particular, \(a\) lies below cofinally many of the \(a_\alpha\).
\end{proof}

\begin{thm}
Assume \(\AC_\kappa\). Let \(\sS=\tup{\bbP,\sG,\sF}\) be a symmetric system and let \(G\subseteq \bbP\) be \(V\)-generic. If \(\sS\) is \(\kappa\)-descending distributive then \(V[G]\) has no functions on \(\kappa\) that are fresh over \(V[G]_\sS\) (or \(V\)).
\end{thm}

\begin{proof}
By \cref{thm:generous-implies-no-fresh}, \(V[G]\) has no fresh functions on \(\kappa\) over \(V\), so let us prove that there are also no fresh functions on \(\kappa\) over \(V[G]_\sS\). Let \(\ddf\) be a \(\bbP\)-name and \(p \in \bbP\) be such that \(p \forces_\bbP``\ddf\colon\check{\kappa} \to V[G]_\sS\) is fresh.'' As in the proof of \cref{thm:generous-implies-no-fresh}, we may obtain \(q \leq p\) such that \(q\) decides the values of all \(\ddf \res \check{\alpha}\). Use \(\AC_\kappa\) to obtain a sequence \(\tup{\ddg_\alpha \mid \alpha < \kappa}\) of \(\sS\)-names such that \(q \forces_\bbP \ddg_\alpha = \ddf\res\check{\alpha}\) for all \(\alpha < \kappa\). By the \(\kappa\)-descending distributivity of \(\sF\), there is unbounded \(I \subseteq \kappa\) such that \(\bigcap_{\alpha \in I}\sym(\ddg_\alpha) \in \sF\). This witnesses that \(\ddx = \Set{\ddg_\alpha \mid \alpha \in I}^\bullet\) is an \(\sS\)-name. \(q \forces_\bbP \bigcup\ddx = \ddf\), so \(q \forces_\bbP \ddf \in V[G]_\sS\) as required.
\end{proof}

\subsection{Combining descending distributivity}

We shall later wish to take a large product of descending distributive notions of forcing while retaining this descending distributivity, originally motivating \cref{thm:generous-lindenbaum-product}. However, similarities between that result and Hamkins's \emph{Key Lemma} (\cref{lem:hamkins-key}) motivated the iterative form \cref{thm:generous-lindenbaum-iteration}.

\begin{thm}\label{thm:generous-lindenbaum-iteration}
If \(\kappa \geq \aleph^\ast(\bbP)\) is regular and \(\1 \forces_\bbP ``\ddbbQ\) is \(\check{\kappa}\)-descending distributive'' then \(\bbP \iter \ddbbQ\) is \(\kappa\)-descending distributive.
\end{thm}

\begin{proof}
Let \(\calD = \tup{D_\alpha \mid \alpha < \kappa}\) be a decreasing sequence of open dense subsets of \(\bbP \iter \ddbbQ\). Since \(\bigcap \calD\) is open, it remains to show that it is dense, so let \(\tup{p_\ast,\ddq_\ast} \in \bbP \iter \ddbbQ\). Note that, for \(V\)-generic \(G \subseteq \bbP\), the sets \(D_\alpha^G = \Set{\ddq^G \mid (\exists p \in \bbP) \tup{p,\ddq} \in D_\alpha}\) are open dense in \(\bbQ\). For if \(q \in \bbQ\), say \(q = \ddq^G\), then for all \(p_0 \in \bbP\) there is \(\tup{p,\ddq'} \in D_\alpha\) with \(\tup{p,\ddq'} \leq \tup{p_0,\ddq}\). In particular, the set of \(p \in \bbP\) forcing that an element that extends \(\ddq\) appears in \(D_\alpha^G\) is dense, so there is indeed \(q' \in D_\alpha^G\) extending \(q\). Let \(G \subseteq \bbP\) be \(V\)-generic with \(p_\ast \in G\).

Since the sets \(D_\alpha\) are decreasing, so too are the sets \(D_\alpha^G\). For if \(q \in D_\alpha^G\), witnessed by \(\tup{p,\ddq} \in D_\alpha\) and \(p \in G\), then for all \(\beta < \alpha\) we have \(\tup{p,\ddq} \in D_\beta\) as well, so certainly \(q \in D_\beta^G\). Since \(\1 \forces_\bbP ``\ddbbQ\) is \(\check{\kappa}\)-descending distributive'', we have that \(D^G = \bigcap_{\alpha < \kappa} D_\alpha^G\) is open dense in \(\bbQ\). In particular, there is \(q \in D^G\) with \(q \leq \ddq_\ast^G\). Let \(\ddq\) be a \(\bbP\)-name for \(q\).

Let \(p \in G\) with \(p \leq p_0\) be such that \(p \forces_\bbP \ddq \in D^G\). Then for all \(\alpha < \kappa\) there are densely many \(p' \leq p\) such that \(\tup{p',\ddq} \in D_\alpha\) (for if \(p' \forces_\bbP \ddq' = \ddq\) and \(\tup{p',\ddq'} \in D_\alpha\) then \(\tup{p',\ddq} \leq \tup{p', \ddq} \in D_\alpha\), so \(\tup{p', \ddq} \in D_\alpha\)). For \(p' \leq p\), define the rank of \(p'\) to be \(\rnk(p') = \sup\Set{\alpha < \kappa \mid \tup{p',\ddq} \in D_\alpha}\). Since \(\kappa \geq \aleph^\ast(\bbP)\), \(\kappa\) is regular and \(\sup\Set{\rnk(p') \mid p' \leq p} = \kappa\), there must be \(p' \leq p\) such that \(\rnk(p') = \kappa\). In this case, \(\tup{p',\ddq} \in D_\alpha\) for all \(\alpha\). Since \(p' \leq p\) we have that \(p' \forces_\bbP \ddq \leq \ddq_\ast\) and so \(\tup{p',\ddq} \leq \tup{p_\ast,\ddq_\ast}\) as required.
\end{proof}

\begin{lem}[Hamkins]\label{lem:hamkins-key}
``If \(\abs{\bbP} \leq \beta\) [is atomless], \(\forces \ddbbQ\) is \({\leq}\beta\)-strategically closed and \(\operatorname{cof}(\theta) > \beta\), then \(\bbP \iter \ddbbQ\) adds no fresh \(\theta\)-sequences'' \textup{\cite[p.~240]{hamkins_gap_2001}.\footnote{A definition of strategic closure is provided in \cite{hamkins_gap_2001}.}}
\end{lem}

While the following proposition regarding products of descending distributive notions of forcing can likely be expressed as a corollary of \cref{thm:generous-lindenbaum-iteration}, the proof is at least as long as the direct poof that we present.

\begin{prop}\label{thm:generous-lindenbaum-product}
If \(\kappa \geq \aleph^\ast(\bbP)\) is regular and \(\bbQ\) is \(\kappa\)-descending distributive then \(\bbP \times \bbQ\) is \(\kappa\)-descending distributive.
\end{prop}

\begin{proof}
Let \(\calD = \tup{D_\alpha \mid \alpha < \kappa}\) be a collection of open dense subsets of \(\bbP \times \bbQ\) and let \(\tup{p_0,q_0} \in \bbP\times\bbQ\). For \(\alpha < \kappa\), let \(D_\alpha^0 = \Set{q \in \bbQ \mid (\exists p \leq p_0) \tup{p,q} \in D_\alpha}\), noting that \(D_\alpha^0\) is open dense. Let \(q \leq q_0\) be such that \(\rnk_{(D_\alpha^0)_\alpha}(q) = \kappa\). Then \(\sup\Set{\rnk_\calD(\tup{p,q}) \mid p \leq p_0} = \kappa\), so since \(\kappa \geq \aleph^\ast(\bbP)\) is regular, there must be \(p \leq p_0\) such that \(\rnk_\calD(\tup{p,q}) = \kappa\). Since \(\tup{p,q} \leq \tup{p_0,q_0}\) we are done.
\end{proof}

\begin{qn}\label{qn:product-of-generous}
If \(\bbP\) and \(\bbQ\) are both \(\kappa\)-descending distributive, must \(\bbP \times \bbQ\) also be?
\end{qn}

Preservation is much easier for filters.

\begin{prop}\label{thm:generous-filter-product}
If \(\sF_0\) and \(\sF_1\) are \(\kappa\)-descending distributive filters then so is \(\sF_0 \times \sF_1\).
\end{prop}

\begin{proof}
By \cref{prop:generosity-is-cofinal-filters}, we may assume that \(\kappa\) is regular without loss of generality. Let \(\tup{\tup{a_\alpha, b_\alpha} \mid \alpha < \kappa}\) be a sequence of elements of \(\sF_0 \times \sF_1\). By \(\kappa\)-descending distributivity, there is \(a \in \sF_0\) such that \(a \leq a_\alpha\) for all \(\alpha \in I\), where \(I \subseteq \kappa\) is unbounded. By the \(\kappa\)-descending distributivity of \(\sF_1\), there is \(b \in \sF_1\) and cofinal \(J \subseteq I\) such that, for all \(\alpha \in J\), \(b \leq b_\alpha\). Thus, for all \(\alpha \in J\), \(a_\alpha \times b_\alpha \geq a \times b \in \sF_0 \times \sF_1\).
\end{proof}

In fact, one should really regard \cref{thm:generous-filter-product} as the fact that a product of strongly \(\kappa\)-descending distributive forcings is strongly \(\kappa\)-descending distributive.

\section{Extendable choice}\label{s:levys-axiom}

\begin{defn}
A function \(F\) is \emph{approachable} if its domain is an ordinal \(\alpha\), \(\emptyset\notin F``\alpha\) and, for all \(\beta<\alpha\), \(F\res\beta\) has a selector. That is, for all \(\beta < \alpha\) there is a function \(f\) with domain \(\beta\) such that, for all \(\gamma < \beta\), \(f(\gamma) \in F(\gamma)\). We shall say that \(F\) is \emph{\(\alpha\)-approachable} to further specify that the domain of \(F\) is \(\alpha\).
\end{defn}

In \cite{levy_interdependence_1964} Levy introduces the following axiom that we call the axiom of \emph{extendable choice} (for an ordinal \(\alpha\)).\footnote{In \cite{levy_interdependence_1964}, Levy does not name this axiom and denotes it `\(C(\alpha)\)'. We shall be using our new denotion to avoid confusion with choice principles. Howard--Paul \cite{howard_consequences_1998} notably uses `\(C\)' to denote variants of the axiom of choice.}

\begin{defn}
\emph{Extendable choice} (for \(\alpha\)), denoted \(\AL_\alpha\), is the assertion that every \(\alpha\)-approachable function has a selector.

Given \(X\), we shall write \(\AL_\alpha(X)\) to mean that every approachable function \(F\colon\alpha\to\power(X)\) has a selector.
\end{defn}

\begin{thm}[{\cite[Theorems~1 and 2]{levy_interdependence_1964}}]\label{thm:levy-cc-deductions}\label{thm:levy-al-deductions}
\(\AL_{0}\) holds and, for all ordinals \(\gamma\), \(\AL_{\gamma+1}\) holds. Furthermore, if \(\cf(\gamma)=\cf(\beta)\) then \(\AL_{\gamma}\lra\AL_{\beta}\) holds.
\end{thm}

Levy goes on to demonstrate the optimality of this theorem. He constructs, for all non-zero limit ordinals \(\alpha\) and \(\beta\) of differing cofinality, a \(\ZFA\) model of \(\AL_{\alpha}\land\lnot\AL_{\beta}\). That is, no additional deductions about the truth of \(\AL_{\alpha}\) can be made from the truth of a single other \(\AL_\beta\) except using \cref{thm:levy-cc-deductions}. See \cref{s:levys-axiom;ss:violating-al} for a demonstration of this fact in \(\ZF\). On the other hand, certain large cardinals do allow one to deduce more about the \(\AL\) structure of the universe (see \cref{thm:supercompact-massive-failure}).

\(\AL\) is intimately related to the notion of eccentricity. For if \(\kappa\) is a limit cardinal and \(\AL_{\kappa}\) holds, then there is no set \(X\) such that \(\aleph(X)=\kappa\). Not only this, but since \(\AL_{\kappa}\) is equivalent to \(\AL_{\cf(\kappa)}\), we in fact deduce that for all limits \(\mu\) with \(\cf(\kappa)=\cf(\mu)\) there is no set \(X\) with \(\aleph(X)=\mu\). However, by \cref{thm:going-up-singular-or-limit} (essentially \cite[Lemma~4.18]{ryan-smith_acwo_2024}), if there is \emph{any} limit \(\mu\) with \(\cf(\mu)=\cf(\kappa)\) and a set \(X\) with \(\aleph(X)=\mu\), then there is \(\mu^\ast\) such that for all limits \(\eta\geq\mu^\ast\) with \(\cf(\eta)=\cf(\kappa)\) there is a set \(Y\) with \(\aleph(Y)=\eta\). Combined with \cref{lem:not-cc-gives-eccentricity} (or \cite[Theorem~16]{levy_interdependence_1964}), we obtain the following.

\begin{cor}\label{cor:cc-equivalents}\label{cor:al-equivalents}
Let \(\kappa\) be an infinite cardinal and \(\calC\) be the class of cardinals \(\mu\) with \(\cf(\mu)=\cf(\kappa)\). The following are equivalent:
\begin{enumerate}
\item \(\AL_{\kappa}\);
\item for all \(\mu\in\calC\), \(\AL_{\mu}\);
\item for some \(\mu\in\calC\), \(\AL_{\mu}\);
\item there is a limit \(\mu\in\calC\) and a set \(X\) such that \(\aleph(X)=\mu\);
\item there is \(\mu\in\calC\) and a set \(X\) such that \(\aleph(X)=\mu<\aleph^\ast(X)\); and
\item there is \(\mu^\ast\in\calC\) such that, for all limit cardinals \(\mu\in\calC\setminus\mu^\ast\), there is a set \(X\) such that \(\aleph(X)=\mu\).
\end{enumerate}
\end{cor}

\begin{prop}\label{lem:not-cc-gives-eccentricity}
If \(\AL_{\delta}\) fails then there is a limit \(\kappa\) such that \(\cf(\delta)=\cf(\kappa)\) and a set \(D\) such that \(\aleph(D)=\kappa\).
\end{prop}

\begin{proof}
Given \(F\) witnessing the failure of \(\AL_{\delta}\), apply \cite[Proposition~3.3]{ryan-smith_acwo_2024} to the set \(X=\Set{F(\alpha)\mid\alpha<\delta}\), noting that the resulting set \(D\) must have \(\aleph(D)=\kappa\) because \(X\) has no choice function and that \(\cf(\kappa)=\cf(\delta)\) by construction.
\end{proof}

\subsection{Preservation of \(\AL_{\kappa}\)}

Let us highlight some ways in which \(\AL_{\kappa}\) may be preserved in a symmetric extension.

\begin{thm}\label{thm:al-preservation}
Assume \(\AL_\kappa\) and let \(\sS\) be a \(\kappa\)-descending distributive symmetric system. Then \(\1 \forces_\sS \AL_{\check{\kappa}}\).
\end{thm}

\begin{proof}
By \cref{cor:generosity-is-cofinal,thm:levy-al-deductions}, we may assume that \(\kappa\) is regular. Let \(p \forces_\sS `` \ddF\) is \(\check{\kappa}\)-approachable.'' By symmetric mixing, for all \(\alpha<\kappa\) there is an \(\sS\)-name \(\ddC_\alpha\) such that \(\sym(\ddC_\alpha)\geq\sym(\ddF)\) and \(p\forces_\sS``\ddC_\alpha\) is the set of selectors for \(\ddF\res\check{\alpha}\).'' Assume also that \(\ddC_\alpha\) is open, so if \(\tup{q,\ddc} \in \ddC_\alpha\) then for all \(r \leq q\), \(\tup{r,\ddc} \in \ddC_\alpha\). By further symmetric mixing, we may assume that for all \(\tup{q,\ddc_\alpha}\in\ddC_\alpha\), \(\ddc_\alpha\) is of the form \(\Set{\tup{\check{\beta},\ddx_{\alpha,\beta}}^\bullet\mid\beta<\alpha}^\bullet\) for various \(\sS\)-names \(\ddx_{\alpha,\beta}\). Let
\begin{equation*}
D_\alpha = \Set{q \in \bbP \mid q \perp p}\cup\Set{q \leq p \mid (\exists\ddc) \tup{q,\ddc} \in \ddC_\alpha},
\end{equation*}
noting that \(D_\alpha\) is open dense. By \(\kappa\)-descending distributivity, let \(q \leq p\) be such that \(\rnk_{(D_\alpha)_\alpha}(q) = \kappa\) and let \(A=\Set{\alpha<\kappa\mid(\exists\ddc)\tup{q,\ddc}\in\ddC_\alpha}\). Since we may restrict any name to an initial segment, we must in fact have that \(A=\kappa\). Furthermore, by this restriction, the function \(\alpha\mapsto\Set{\ddc\mid\tup{q,\ddc}\in\ddC_\alpha}\) is \(\kappa\)-approachable. Therefore, we may select a sequence \(\tup{\ddc_\alpha\mid\alpha<\kappa}\) such that, for all \(\alpha<\kappa\), \(\tup{q,\ddc_\alpha}\in\ddC_\alpha\). Let \(H_\alpha=\sym(\ddc_\alpha)\in\sF\). By the \(\kappa\)-descending distributivity of \(\sF\), there is \(H \in \sF\) such that, for unboundedly many \(\alpha\), \(H \leq H_\alpha\). Thus, \(\Set{\tup{\check{\alpha},\ddc_\alpha}^\bullet\mid H \leq H_\alpha}^\bullet\) is an \(\sS\)-name. Since a cofinal sequence of selectors for \(\ddF\) exists in the symmetric extension, so does a selector for \(\ddF\) itself as required.
\end{proof}

We should remark that descending distributivity will also produce a preservation of the \emph{failure} of \(\AL\).

\begin{thm}\label{thm:al-non-preservation}
Assume \(\lnot\AL_\kappa\) and let \(\bbP\) be \(\kappa\)-descending distributive. Then \(\bbP\) does not add any selectors for ground-model witnesses of \(\lnot \AL_\kappa\). In particular, any symmetric system forcing with \(\bbP\) will also preserve \(\lnot \AL_\kappa\).
\end{thm}

\begin{proof}
By \cref{cor:generosity-is-cofinal,thm:levy-al-deductions}, we may assume that \(\kappa\) is regular. Let \(F_0\colon\kappa\to X\) be approachable, witnessing the failure of \(\AL_\kappa\). Since \(\kappa\) is regular, replace \(F_0\) by \(F\) such that \(F(\alpha)\) is the set of (ground model) selectors for \(F_0\res\alpha\), noting that \(F\) is also \(\kappa\)-approachable. Suppose, for a contradiction, that for some \(\sS\)-name \(\ddf\) and \(p\in\bbP\) we have that \(p\forces_\sS``\ddf\) is a selector for \(\check{F}\).'' Define the open dense set \(D_\alpha = \Set{q \in \bbP \mid q \perp p}\cup\Set{q \leq p \mid (\exists c) q\forces_\sS \ddf(\check{\alpha}) = \check{c}}\). By \(\kappa\)-descending distributivity, let \(I \subseteq \kappa\) be unbounded and \(q \leq p\) be such that \(q\) decides \(\ddf(\check{\alpha}) = \check{c}_\alpha\) for all \(\alpha \in I\). Then \(c \colon \alpha \mapsto c_\beta(\alpha)\), where \(\beta = \min I\setminus(\alpha+1)\), is a selector for \(F_0\), contradicting our assumption.
\end{proof}

\subsection{Violating extendable choice}\label{s:levys-axiom;ss:violating-al}

In \cite{levy_interdependence_1964}, Levy constructs Fraenkel--Mostowski models \(\frakJ_\beta\), where \(\beta\) is an ordinal, to carefully violate various \(\AL_\kappa\).

\begin{prop}[{Levy, \cite[p.~151]{levy_interdependence_1964}}]
If \(\cf(\kappa)\neq\cf(\aleph_\beta)\), then \(\frakJ_\beta\models\AL_\kappa\). If \(\beta\) is a successor then \(\frakJ_\beta\models\AL_{\aleph_\beta}\).
\end{prop}

That is, if one wants to violate \(\AL_\kappa\) only when \(\cf(\kappa)=\cf(\mu)\) for a cardinal \(\mu\), one could take, for example, \(\frakJ_\mu\) as their model. We shall extend this result by instead producing symmetric extensions to achieve this same goal in \(\ZF\), using a modification of the construction from \cite[Theorem~3.1]{karagila_which_2024}.

In the following we make use of \emph{wreath products} of groups. If \(G\) and \(H\) are groups acting faithfully on sets \(X\) and \(Y\) respectively, then \(G \wr H\) is the group of permutations \(\pi\) of \(X \times Y\) such that, for some \(g \in G\) and sequence \(\tup{h_x \mid x \in X} \in H^X\), the action of \(\pi\) is given by \(\pi(x,y) = \tup{g(x), h_x(y)}\). Note that \(g\) and \(\tup{h_x \mid x \in X}\) are uniquely determined by \(\pi\), so we use \(\pi^\ast\) to refer to \(g\) and \(\pi_x\) to refer to \(h_x\). See \cite[Section~2.2]{ryan-smith_upwards_2024} for a more in-depth introduction.

\begin{thm}\label{thm:eliminate-one-al}
Let \(\kappa\) be such that \(\kappa^{{<}\kappa}=\kappa\) is infinite (in particular, \(\kappa\) is regular and \(\kappa^{{<}\kappa}\) is well-orderable). Then there is a cofinality-preserving symmetric extension in which \(\AL_\kappa\) fails. Furthermore, the symmetric extension preserves the truth value of \(\AL_\mu\) for all \(\mu\) not cofinal with \(\kappa\).
\end{thm}

\begin{proof}
We define the symmetric system \(\sS=\tup{\bbP,\sG,\sF}\) as follows:
\begin{enumerate}
\item \(\bbP=\Add(\kappa,\kappa\times\kappa\times\kappa)\);\footnote{That is, the conditions of \(\bbP\) are partial functions \(\kappa^4\to2\) with domain of cardinality less than \(\kappa\), where \(q\leq p\) if \(q\supseteq p\).}
\item \(\sG\leq(\Set{\id}\wr S_\kappa)\wr S_\kappa\) consists of those \(\pi\) such that \(\abs{\supp(\pi^\ast)}<\kappa\), with \(\sG\) acting on \(\bbP\) via \(\pi p(\pi(\alpha,\beta,\gamma),\delta)=p(\alpha,\beta,\gamma,\delta)\);
\item \(\sF\) is generated as a filter by all groups of the form \(H_{J,K}\) for \(J\in[\kappa\times\kappa]^{{<}\kappa}\) and \(K\in[J\times\kappa]^{{<}\kappa}\), where
\begin{equation*}
H_{J,K}=\Set*{\pi\in\sG\mid\pi^\ast\res J=\id,\,\pi\res K=\id}.
\end{equation*}
\end{enumerate}
If \(\pi\in\sG\) and \(H_{J,K}\in\sF\) then \(\pi H_{J,K}\pi^{-1}\geq H_{J',K'}\), where we define \(K'=\pi``K\) and \(J'=(\pi^\ast)``J\cup J\cup\supp(\pi^\ast)\) (the proof is essentially identical to \cite[Claim~3.1.1]{karagila_which_2024}). Thus, \(\sF\) is normal and \(\sS\) is a symmetric system.

We would be remiss to not make the following immediate remarks:
\begin{enumerate}[label=(\roman*)]
\item Since \(\kappa^{{<}\kappa}=\kappa\), \(\bbP\) is well-orderable. Thus, by the usual \(\ZFC\) proof, \(\bbP\) does not change the cofinalities of ordinals.
\item For all \(\eta\), \(\DC_\eta\) holds for all subtrees of \(\bbP^{{<}\eta}\) (again by the well-orderability of \(\bbP\)).
\item \(\sF\) has a basis \(\sB\) of subgroups indexed by \([\kappa\times\kappa]^{{<}\kappa}\times[\kappa\times\kappa]^{{<}\kappa}\), so in particular of cardinality \(\kappa\). Furthermore, \(\sB\) is well-orderable and so we may construct a uniform function \(\sF \to \sB\) sending \(H\) to \(H' \leq H\) with \(H' \in \sB\).
\item \(\bbP\) is \(\kappa\)-distributive and \(\sF\) is \(\kappa\)-complete.
\end{enumerate}
Let \(\mu\neq\kappa\) be regular. If \(\mu>\kappa\) then \(\mu\geq\aleph^\ast(\bbP\cup\sB)=\kappa^+\), so \(\sS\) is \(\mu\)-descending distributive. On the other hand, if \(\mu < \kappa\) then \(\sS\) is \(\mu\)-distributive and so, again, \(\mu\)-descending distributive (recall \cref{prop:common-generosity}). Thus, by \cref{thm:al-preservation,thm:al-non-preservation}, \(\sS\) does not change the truth value of \(\AL_\mu\). Since \(\sS\) is cofinality-preserving, this covers all cases.

The remainder of the proof shall be spent showing that \(\1\forces_\sS\lnot\AL_{\check{\kappa}}\).

For \(\alpha,\beta,\gamma\in\kappa\), let \(\ddy_{\alpha,\beta,\gamma}=\Set{\tup{p,\check{\delta}}\mid p(\alpha,\beta,\gamma,\delta)}\), \(\ddx_{\alpha,\beta}=\Set{\ddy_{\alpha,\beta,\gamma}\mid\gamma<\kappa}^\bullet\) and \(\ddX=\Set{\ddx_{\alpha,\beta}\mid\alpha,\beta<\kappa}^\bullet\). By standard calculations, we see that \(\pi\ddy_{\alpha,\beta,\gamma}=\ddy_{\pi(\alpha,\beta,\gamma)}\), \(\pi\ddx_{\alpha,\beta}=\ddx_{\pi^\ast(\alpha,\beta)}\) and \(\pi\ddX=\ddX\). Thus all of these are \(\sS\)-names. We claim that \(\1\forces_\sS\aleph(\ddX)=\check{\kappa}\land\aleph^\ast(\ddX)=\check{\kappa}^+\).

Note that if \(G\) is \(\sS\)-generic for \(V\), then \(V[G]\models\abs{\ddX^G}=\kappa\). Therefore there are no surjections \(X \to \kappa^+\) in \(V[G]\) and hence \(\aleph(X)^{V[G]_\sS}\leq\aleph^\ast(X)^{V[G]_\sS}\leq\kappa^+\).

Let \(\ddf=\Set{\tup{\ddx_{\alpha,\beta},\check{\alpha}}^\bullet\mid\alpha,\beta<\kappa}^\bullet\). For \(\pi \in \sG\),
\begin{align*}
\pi\ddf&=\Set*{\tup{\pi\ddx_{\alpha,\beta},\pi\check{\alpha}}^\bullet\mid\alpha,\beta<\kappa}^\bullet\\
&=\Set*{\tup{\ddx_{\pi^\ast(\alpha,\beta)},\check{\alpha}}^\bullet\mid\alpha,\beta<\kappa}^\bullet\\
&=\Set*{\tup{\ddx_{\alpha,\pi^\ast_\alpha(\beta)},\check{\alpha}}^\bullet\mid\alpha,\beta<\kappa}^\bullet\\
&=\ddf
\end{align*}
and so \(\ddf\) is an \(\sS\)-name for a surjection \(\ddX\to\check{\kappa}\). Hence \(\1\forces_\sS\check{\kappa}<\aleph^\ast(\ddX)\). Combined with our earlier argument, we see that \(\1\forces_\sS\aleph^\ast(\ddX)=\check{\kappa}^+\).

For \(\alpha<\kappa\), let \(\ddg_\alpha=\Set{\tup{\check{\beta},\ddx_{0,\beta}}^\bullet\mid\beta<\alpha}^\bullet\). Then \(\sym(\ddg_\alpha)\geq H_{\Set{0}\times\alpha,\emptyset}\) and so is an \(\sS\)-name for an injection \(\check{\alpha}\to\ddX\). Hence \(\1\forces_\sS\aleph(\ddX)\geq\check{\kappa}\). It remains to show that \(\1\forces_\sS\aleph(\ddX)\leq\check{\kappa}\).

Suppose that for some \(p\in\bbP\) and \(\sS\)-name \(\ddh\), \(p\forces_\sS\ddh\colon\check{\kappa}\to\ddX\). Let \(J\) and \(K\) be such that \(\sym(\ddh)\geq H_{J,K}\). We shall show that \(p\forces_\sS\ddh``\check{\kappa}\subseteq\Set{\ddx_{\alpha,\beta}\mid\tup{\alpha,\beta}\in J}^\bullet\). If not then for some \(q\leq p\), \(\tup{\alpha,\beta}\notin J\) and \(\eta<\kappa\) we have \(q\forces_\sS\ddh(\check{\eta})=\ddx_{\alpha,\beta}\). Let \(\beta'<\kappa\) be such that \(\tup{\alpha,\beta'}\notin J\).

We shall produce \(\pi\in H_{J,K}\) such that \(\pi^\ast(\alpha,\beta)=\tup{\alpha,\beta'}\) and \(\pi q\comp q\).\footnote{This is essentially \cite[Claim~3.1.2]{karagila_which_2024}.} Let \(A=\Set{\gamma<\kappa\mid\tup{\alpha,\beta,\gamma}\in\supp(q)}\) and \(B=\Set{\gamma<\kappa\mid\tup{\alpha,\beta',\gamma}\in\supp(q)}\). Since \(\abs{A},\abs{B}<\kappa\), there is \(\sigma\in S_\kappa\) such that \(\sigma``A\cap B=\emptyset\) and \(A\cap\sigma``B=\emptyset\). Define \(\pi\in\sG\) by setting \(\pi^\ast\) to be the transposition \((\tup{\alpha,\beta}\ \tup{\alpha,\beta'})\), setting \(\pi_{\alpha,\beta}=\pi_{\alpha,\beta'}=\sigma\) and taking \(\pi_{\alpha'',\beta''}=\id\) for all other \(\tup{\alpha'',\beta''} \in \kappa \times \kappa\). Then \(\pi q\comp q\), \(\pi\in H_{I,J}\) and \(\pi^\ast(\alpha,\beta)=\tup{\alpha,\beta'}\) as required.

Note that \(\pi q\forces_\sS\ddh(\check{\eta})=\ddx_{\alpha,\beta'}\), so \(q\cup\pi q\forces_\sS\ddh(\check{\eta})=\ddx_{\alpha,\beta'}=\ddx_{\alpha,\beta}\), a contradiction. Hence \(\ddh\) cannot be forced to be an injection and \(\1\forces_\sS\aleph(\ddX)\leq\check{\kappa}\) as required.
\end{proof}

By taking products of these symmetric extensions one can tailor exactly for which \(\kappa\) we have \(\AL_\kappa\) holding, up to the requirements of \cref{thm:levy-al-deductions}. In the following, we do not put much thought into what constitutes a `class'. For our purposes, it is sufficient that the class \(\calC\subseteq\Ord\) satisfies that for all \(\alpha\in\Ord\), \(\calC\cap\alpha\in V\) and that forcing with the Easton-support product of \(\Add(\alpha,1)\) for all \(\alpha\in\calC\) preserves \(\ZFC\).\footnote{We cannot have \(\calC\) be a cofinal sequence of ordinals of order type \(\omega\), for example.} Demanding that \(\calC\) is definable is sufficient.

\begin{thm}\label{thm:new-models}
Assume \(\GCH\) and let \(\calC\) be a class of infinite regular cardinals such that, whenever \(\kappa\) is inaccessible and \(\kappa \cap \calC\) is unbounded below \(\kappa\), we have \(\kappa \in \calC\).\footnote{We shall show in \cref{lem:weirdness-required} that our construction does in fact require that \(\calC\) is closed under regular limit points. However, we do not know if `\(\kappa\) is weakly inaccessible and for all regular, non-zero \(\alpha < \kappa\), \(\lnot\AL_\alpha\)' implies \(\lnot\AL_\kappa\) (see \cref{qn:do-we-need-weirdness})} Then there is a symmetric extension \(M\) such that
\begin{equation*}
\Spec_\aleph(M)=\Set*{\tup{\kappa^+,\kappa^+}\mid\kappa\in\Card}\cup\Set*{\tup{\kappa,\kappa^+}\mid\cf(\kappa)\in\calC}.
\end{equation*}
In particular, \(M\models(\forall\kappa)\AL_{\kappa}\lra\cf(\kappa)\notin\calC\).
\end{thm}

\begin{proof}
For \(\alpha\in\Ord\), if \(\alpha\notin\calC\) then let \(\sS_\alpha\) be the trivial symmetric system \(\tup{\Set{\1},\Set{e},\Set{\Set{e}}}\). If \(\alpha\in\calC\), then let \(\sS_\alpha\) be the symmetric system given by \cref{thm:eliminate-one-al} and let \(\sS\) be the Easton-support product of all these \(\sS_\alpha\). Let \(G\) be \(\sS\)-generic and \(M=V[G]_\sS\). For \(\alpha\in\Ord\), let \(G\res\alpha\) be the restriction of \(G\) to the co-ordinates below \(\alpha\) and \(M_\alpha=V[G\res\alpha]_\sS=V[G\res\alpha]_{\sS\res\alpha}\). Note that \(M=\bigcup\Set{M_\alpha \mid \alpha \in \Ord}\). For \(A\) a class of ordinals we shall similarly use \(\sS \res A\) to denote the restriction of \(\sS\) (or \(\bbP\), \(\sG\), \(\sF\)) to the co-ordinates in \(A\). We denote by \(\bbP^{{\geq}\alpha}\) the restriction \(\bbP\res\Ord\setminus\alpha\). We shall complete the proof by demonstrating the following:
\begin{enumerate}
\item \(M\models\ZF\);
\item for all regular \(\kappa\notin\calC\), \(M\models\AL_\kappa\);
\item for all \(\kappa \in \calC\) there is \(X\in M\) such that \(M\models\aleph(X)=\kappa\land\aleph^\ast(X)=\kappa^+\);
\item \(M\models(\forall X)\aleph^\ast(X)\leq\aleph(X)^+\).
\end{enumerate}
Between these three steps, \cref{cor:al-equivalents,thm:going-up-regular}, we obtain the result.

\textit{Step}\enspace{}1.\quad{}We shall firstly remark that \(V[G]\models\ZFC\) since, for all regular cardinals \(\kappa\), \(\bbP=\bbP\res\kappa\times\bbP^{{\geq}\kappa}\), where \(\bbP\res\kappa\) is a set with the \(\kappa^{++}\)-chain condition and \(\bbP^{{\geq}\kappa}\) is \(\kappa\)-closed.\footnote{See \cite[Theorem~15.18]{jech_set_2003} for details. Note that \(\bbP\res\kappa\) will usually have the \(\kappa^+\)-chain condition, but regular limit \(\kappa\notin\calC\) will fail this.} By the construction of symmetric extensions, \(M\) is a transitive class between \(V\) and \(V[G]\), so per \cref{thm:inner-model} it is sufficient to prove that it is almost universal and closed under G\"odel operations (similar to \cite[Theorems~5.6 or 9.4]{karagila_iterating_2019}).

\textit{Almost universality}.\quad{}Recall that \(M\) is \emph{almost universal} if for all sets \(x\subseteq M\) there is \(y\in M\) with \(y\supseteq x\). We shall do this by showing that for all \(\eta\) there is \(F(\eta)\in\Ord\) such that, for all \(\beta\geq F(\eta)\), no sets of von Neumann rank below \(\eta\) are added by \(\bbP^{{\geq}\beta}\). However, taking \(F(\eta)=\abs{\eta}^+\) is sufficient, since \(\bbP^{{\geq}\abs{\eta}^+}\) is \(\abs{\eta}^+\)-closed. Hence, the \emph{a priori} class of \(\sS\)-names of von Neumann rank below \(\eta\) is in fact a set. Taking \(\alpha\) large enough that \(\bbP \res \beta \in V_\alpha\), we then have that \(\ddX_\eta=\Set{\ddx\mid \ddx\in V_{\alpha+\eta}\cap V^\sS}^\bullet\) is fixed by all permutations of \(\bbP\), so \(\ddX_\eta^G\supseteq V_\eta\cap M\) and \(\ddX_\eta^G\in M\).

\textit{G\"odel operations}.\quad{}The constructions of the G\"odel operations is highly standard and we shall prove only one, \(G_4(X,Y)=X\setminus Y\), to demonstrate the method. If \(X,Y\in M\) then there are \(\sS\)-names \(\ddX,\ddY\) such that \(\ddX^G=X\) and \(\ddY^G=Y\). Then let \(\ddZ=\Set{\tup{p,\ddu}\mid p\in\bbP,\ddu\in V^\sS,p\forces_\sS\ddu\in\ddX\land\ddu\notin\ddY}\). By the symmetry lemma, if \(\pi\in\sym(\ddX)\cap\sym(\ddY)\) and \(\tup{p,\ddu}\in\ddZ\) then \(\pi p\forces_\sS\pi\ddu\in\ddX\land\pi\ddu\notin\ddY\), so \(\tup{\pi p,\pi\ddu}\in\ddZ\). By restricting the \(\sS\)-names \(\ddu\) to only those of von Neumann rank below that of \(X\cup Y\), we obtain a set \(\ddZ_\ast\) that is still an \(\sS\)-name and still satisfies \(\ddZ^G=X\setminus Y\) as required.

At this point it behoves us to point out that \(M\) has the same cofinalities and cardinalities on ordinals as \(V\) does. The standard proof of Easton's theorem\footnote{From \cite{easton_powers_1970}, though \cite[Theorem~15.18]{jech_set_2003} also provides a good treatment.} shows that every regular \(\kappa\) in \(V\) will still be regular in \(V[G]\).

\textit{Step}\enspace{}2.\quad{}Note that since \(M=\bigcup\Set{M_\alpha \mid \alpha\in\Ord}\), it is sufficient to prove that, for all \(\alpha\in\Ord\), \(M_\alpha\models\AL_\kappa\). For notational convenience we shall denote by \(\sT_0\) the symmetric system \(\sS\res\kappa\) and by \(\sT_1\) the symmetric system \(\sS\res\Set{\beta \mid \kappa \leq \beta < \alpha}\). We shall extend this notation so that \(\sT_0 = \tup{\bbP_0,\sG_0,\sF_0}\) and \(\sT_1 = \tup{\bbP_1,\sG_1,\sF_1}\).

If \(\kappa = \eta^+\) then \(\abs{\bbP_0},\abs{\sF_0} \leq \eta < \kappa\). On the other hand, if \(\kappa\) is a limit then by assumption \(\kappa \cap \calC\) is bounded below \(\kappa\) and so again \(\abs{\bbP_0},\abs{\sF_0} < \kappa\). In particular, we have \(\kappa \geq \aleph^\ast(\bbP_0 \cup \sF_0)\). Note also that \(\bbP_1\) is \(\kappa^+\)-closed (and hence \(\kappa^+\)-distributive) and that \(\sF_1\) is \(\kappa^+\)-complete. Thus, by \cref{thm:generous-lindenbaum-product,thm:generous-filter-product}, \(\bbP_0 \times \bbP_1\) and \(\sF_0 \times \sF_1\) are \(\kappa\)-descending distributive and so, by \cref{thm:al-preservation}, \(\sS \res \alpha = \sT_0 \times \sT_1\) preserves \(\AL_\kappa\).

\textit{Step}\enspace{}3.\quad{}Note that for all \(\kappa\in\calC\), \(\sS_\kappa\) adds a set \(X\) such that \(\aleph(X)=\kappa\) and \(\aleph^\ast(X)=\kappa^+\), per \cref{thm:eliminate-one-al}. Therefore \(M\models\aleph(X)\geq\kappa\land\aleph^\ast(X)\geq\kappa^+\), so it suffices to prove that \(M\models\aleph(X)\leq\kappa\land\aleph^\ast(X)\leq\kappa^+\). However, in \(V[G\res\mu]\models\ZFC\) we have collapsed no cardinals and have that \(\abs{X}=\kappa\), so \(V[G]\models\aleph^\ast(X)\leq\kappa^+\) and so certainly \(M\models\aleph^\ast(X)\leq\kappa^+\). Thus it remains to show that \(M \models \aleph(X) \leq \kappa\) and, similar to Step 2, it is sufficient to show that, for all \(\alpha>\kappa\), \(M_\alpha\models\aleph(X) \leq \kappa\).

Let \(\ddX\) be the name for \(X\) as constructed in \cref{thm:eliminate-one-al} and suppose that for some \(p \in \bbP \res \alpha\) and \(\sS \res \alpha\)-name \(\ddh\) that \(p\forces_{\sS\res\alpha}``\ddh\) is an injection \(\check{\kappa}\to\ddX\)''. Let \(H=\prod\Set{H_{J_\gamma,K_\gamma}\mid\gamma<\alpha}=\sym(\ddh)\). Now we may use the exact same argument as in \cref{thm:eliminate-one-al}, acting only on the \(\kappa\) co-ordinate, to show that the image of \(\ddh\) must be contained in \(\Set{\ddx_{\delta,\varep}\mid\tup{\delta,\varep}\in J_\kappa}^\bullet\). If we do so then, since \(\sS\) collapses no cardinals, \(\ddh\) cannot be an injection. If \(q\forces_{\sS\res\alpha}\ddh(\check{\eta})=\ddx_{\delta,\varep}\) for some \(\tup{\delta,\varep}\notin J_\kappa\), then let \(\tup{\delta,\varep'}\in J_\kappa\) with \(\varep'\neq\varep\) and construct \(\pi_0 \in H_{J_\kappa,K_\kappa}\) as in \cref{thm:eliminate-one-al} so that \(\pi_0^\ast(\delta,\varep)=\tup{\delta,\varep'}\) and \(\pi_0 q(\kappa) \comp q(\kappa)\). If we then extend \(\pi_0\) to \(\pi \in \sG\res\alpha\) by setting \(\pi\) to be the identity on all co-ordinates other than \(\kappa\) we have that \(\pi \in H\) and \(\pi q \comp q\), so \(q \cup \pi q \forces_{\sS \res \alpha}\ddh(\check{\eta})=\ddx_{\delta,\varep}=\ddx_{\delta,\varep'}\), a contradiction.

\textit{Step}\enspace{}4.\quad{}Suppose that for some \(\alpha\in\Ord\), \(p\in\bbP\), cardinal \(\kappa\) and \(\sS \res \alpha\)-names \(\ddX\) and \(\ddf\), \(p\forces_{\sS\res\alpha}``\ddf\) is a surjection \(\ddX\to\check{\kappa}^+\). We shall show that \(p\forces_{\sS\res\alpha}\aleph(\ddX)>\check{\kappa}\). That is, if \(M\models\aleph^\ast(X)>\kappa^+\) then \(M\models\aleph(X)>\kappa\) as required. Similar to Step 2, we separate \(\sS\) into parts, though this time we separate into three parts, rather than two. We write \(\sT_0\) for \(\sS\res\kappa^+\), \(\sS_\ast\) for \(\sS_{\kappa^+}\), and \(\sT_1\) for \(\sS\res\Set{\beta \in \Ord \mid \kappa^+ < \beta < \alpha}\). By using the \(\kappa^{++}\)-closure of \(\bbP_1\) and the \(\kappa^+\)-closure of \(\bbP_\ast\), we may find \(q_1 \leq p_1\), a descending chain \(\tup{q_\ast^{(\gamma)} \mid \gamma < \kappa^+}\) of conditions in \(\bbP_\ast\) below \(p_\ast\) and a sequence \(\tup{\ddx_{\gamma} \mid \gamma < \kappa^+}\) of \(\sS\res\alpha\)-names such that, for all \(\gamma < \kappa^+\), there is \(q_0 \leq p_0\) such that \(\tup{q_0,q_\ast^{(\gamma)},q_1}\forces_{\sS\res\alpha}\ddf(\ddx_\gamma)=\check{\gamma}\). Now, since \(\abs{\bbP_0}<\kappa^+\), there must be a single \(q_0 \leq p_0\) such that for cofinally many \(\gamma < \kappa^+\), \(\tup{q_0,q_\ast^{(\gamma)},q_1}\forces_{\sS \res \alpha}\ddf(\ddx_\gamma)=\check{\gamma}\). Using the \(\kappa^+\)-completeness of \(\sF_1\) and the fact that \(\abs{\sF_0}<\kappa^+\), we find \(H_0\in\sF_0\) and \(H_1\in\sF_1\) such that for cofinally many \(\gamma < \kappa^+\) there is \(H_\ast \in \sF_\ast\) such that \(\sym(\ddx_\gamma) \geq H_0 \times H_\ast \times H_1\).

However, \(\kappa^+\) is regular and so there are \(\kappa^+\)-many such \(\ddx_\gamma\). Let \(A\) be a set of \(\kappa\)-many \(\gamma<\kappa^+\) such that for some \(H_\ast^\gamma \in \sF_\ast\) we have \(\sym(\ddx_\gamma) \geq H_0 \times H_\ast^\gamma \times H_1\) and \(\tup{q_0,q_\ast^{(\gamma)},q_1} \forces_{\sS \res \alpha} \ddf(\ddx_\gamma) = \check{\gamma}\). In this case, by the \(\kappa^+\)-completeness of \(\sF_\ast\), \(H_\ast \bigcap\Set{ H_\ast^\gamma \mid \gamma \in A} \in \sF_\ast\). Let \(\zeta = \sup A < \kappa^+\). Since, for all \(\gamma \in A\),
\begin{equation*}
q = \tup{q_0, q_\ast^{(\zeta)}, q_1} \forces_{\sS\res\alpha} \ddf(\ddx_\gamma) = \check{\gamma},
\end{equation*}
we have that \(q \forces_{\sS\res\alpha} \ddx_\gamma \in \ddX\). Thus, enumerating \(A\) as \(\tup{\gamma_\delta \mid \delta < \kappa}\), we have that
\begin{equation*}
\ddh = \Set*{ \tup{\check{\delta},\ddx_{\gamma_\delta}}^\bullet \mid \delta < \kappa}^\bullet
\end{equation*}
is an \(\sS\res\alpha\)-name such that \(q\forces_{\sS\res\alpha}``\ddh\) is an injection \(\check{\kappa}\to\ddX\)'' as required. Note that it is injective as for all \(\delta < \delta' < \kappa\), \(q\forces_{\sS\res\alpha}\ddf(\ddx_{\gamma_\delta}) \neq \ddf(\ddx_{\gamma_{\delta'}})\).
\end{proof}

The following proposition demonstrates that our requirement in \cref{thm:new-models} of the closure of \(\calC\) is under regular limits is in fact necessary when using the construction provided.

\begin{prop}\label{lem:weirdness-required}
Assuming \(\GCH\), if \(\calC \cap \kappa\) is unbounded below \(\kappa\), then the symmetric system from \cref{thm:new-models} forces \(\lnot \AL_\kappa\) (even if \(\kappa \notin \calC\)).
\end{prop}

\begin{proof}
We may assume that \(\calC \subseteq \kappa\) since any remaining symmetric extensions are \(\kappa^+\)-distributive and so, by \cref{thm:al-non-preservation}, will not restore \(\AL_\kappa\) once violated. We shall in fact show that \(\sS\) introduces a set \(X\) such that \(\aleph(X) = \kappa\) and \(\aleph^\ast(X) = \kappa^+\).

For \(\delta\in\calC\) we shall write \(\ddy_{\alpha,\beta,\gamma}^\delta\), \(\ddx_{\alpha,\beta}^\delta\) and \(\ddX^\delta\) for the \(\sS_\delta\)-names from the proof of \cref{thm:eliminate-one-al}. That is:
\begin{align*}
\ddy_{\alpha,\beta,\gamma}^\delta&=\Set*{\tup{p,\check{\varep}}\mid p(\delta)(\alpha,\beta,\gamma,\varep)=1}\\
\ddx_{\alpha,\beta}^\delta&=\Set*{\ddy_{\alpha,\beta,\gamma}^\delta\mid\gamma<\delta}^\bullet\\
\ddX^\delta&=\Set{\ddx_{\alpha,\beta}^\delta\mid\alpha,\beta<\delta}^\bullet.
\end{align*}
Furthermore, denote by \(\ddX_0^\delta\) the name \(\Set{\ddx_{0,\beta}^\delta\mid\beta<\delta}^\bullet\). Increasingly enumerate \(\calC\) as \(\Set{\delta_\varep\mid\varep<\kappa}\). Let \(\ddX = \bigcup_{\varep < \kappa} \ddX_0^{\delta_\varep}\). We shall show that \(\ddX\) is the \(\sS\)-name for the desired set.

Firstly, \(\sym(\ddX_0^\delta) = \sG\) for all \(\delta\), so certainly \(\ddX\) is an \(\sS\)-name. For \(\varep < \kappa\), let \(\ddf_\varep = \Set{\tup{\check{\eta},\ddx_{0,\eta}^{\delta_{\varep+1}}}^\bullet \mid \eta < \delta_\varep}^\bullet\). Then \(\sym(\ddf_\varep) = H_{\Set{0}\times\delta_\varep,\emptyset}\prod_{\delta \neq \delta_{\varep+1}} \sG_\delta \in \sF\), so \(\ddf_\varep\) is an \(\sS\)-name for an injection \(\delta_\varep \to \ddX\). Hence \(\1 \forces_\sS \aleph(\ddX) \geq \check{\kappa}\). On the other hand, let \(\ddg = \Set{\tup{\ddx_{0,\eta}^{\delta_\varep}, \check{\varep}}^\bullet \mid \varep < \kappa,\eta<\delta_\varep}^\bullet\). Then \(\sym(\ddg) = \sG\), so \(\ddg\) is an \(\sS\)-name for a surjection \(\ddX \to \check{\kappa}\) and \(\1 \forces_\sS \aleph^\ast(\ddX) \geq \check{\kappa}^+\).

It remains to show that these inequalities are, in fact, equalities. Note that \(\abs{\ddX} = \kappa\), so certainly \(\aleph^\ast(\ddX^G) < \kappa^{++}\). Thus, \(\1 \forces_\sS \aleph^\ast(\ddX) = \check{\kappa}^+\) as desired. On the other hand, suppose that \(\ddf\) is an \(\sS\)-name and \(p \forces_\sS \ddf \colon \check{\kappa} \to \ddX\). Let \(\sym(\ddf) = \prod_{\delta \in \calC} H^{(\delta)}\) and let \(\zeta < \kappa\) be large enough that \(H^{(\delta)} = \sG_\delta\) for all \(\delta \geq \delta_\zeta\). We shall show that \(p \forces_\sS \ddf``\check{\kappa} \subseteq \Set{\ddx_{0,\eta}^{\delta_\varep} \mid \eta < \delta_\varep, \varep < \zeta}^\bullet\). Suppose otherwise, that for some \(q \leq p\), \(\alpha < \kappa\), \(\varep \geq \zeta\) and \(\eta < \delta_\varep\) we have that \(q \forces_\sS \ddf(\check{\alpha}) = \ddx_{0,\eta}^{\delta_\varep}\). Then, just as in the proof of \cref{thm:eliminate-one-al}, we may work over only the \(\delta_\varep\) co-ordinate to produce \(\pi = \prod_{\delta \in \calC} \pi^{(\delta)} \in \sG\) such that \(\pi(\delta) = \id\) for all \(\delta \neq \delta_\varep\), \(q \comp \pi q\) and \(\pi\ddx_{0,\eta}^{\delta_\varep} \neq \ddx_{0,\eta}^{\delta_\varep}\). Then \(q \cup \pi q\) forces that \(\ddf\) is not a function, contradicting our assumption. Thus in fact \(\1 \forces_\sS \aleph(\ddX) \leq \check{\kappa}\) as required.
\end{proof}

Note that in all further extensions we must still have that \(\aleph^\ast(X) = \kappa^+\). Furthermore, any injection \(\kappa \to X\) would be a choice function for the approachable function \(F \colon \alpha \mapsto \Set{\text{injections }\alpha \to X}\), which cannot be introduced by \cref{thm:al-non-preservation}. Hence, in the final model we still have that \(\aleph(X) = \kappa\) and \(\aleph^\ast(X) = \kappa^+\).

\begin{qn}\label{qn:do-we-need-weirdness}
Let \(\kappa\) be weakly inaccessible (that is, a regular limit). Does `for all regular, non-zero \(\alpha<\kappa\), \(\lnot \AL_\alpha\)' imply \(\lnot \AL_\kappa\)?
\end{qn}

Note that the dual to this question is false, it is not the case that \(\AL_{{<}\kappa}\) implies \(\AL_\kappa\).\footnote{\cite[Theorem~4]{jech_interdependence_1966} (or \cite[Theorem~8.3]{jech_axiom_1973}) produces a model of \(\DC_{{<}\kappa}\land\lnot\AC_\kappa\) without altering the weak inaccessibility of \(\kappa\).} Furthermore, there are large cardinal properties that would cause \(\kappa\) to be a continuity point of \(\AL\); supercompact is sufficient, see \cref{s:levys-axiom;ss:large-cardinals}. This motivates the following contortion of \cref{qn:do-we-need-weirdness}.

\begin{qn}\label{qn:large-cardinal-continuity}
What is the weakest large cardinal assumption on \(\kappa\) required to make \cref{qn:do-we-need-weirdness} true? For example, is it sufficient to assume that \(\kappa\) is critical?\footnote{Critical cardinals are introduced in \cite{hayut_critical_2020}.}
\end{qn}

The technique of Step~4 in the proof of \cref{thm:new-models} deserves to be extracted for individual examination, as it provides a much-needed way to prove that certain symmetric extensions do not produce wildly eccentric sets.

\begin{prop}\label{prop:preserve-oblateness-mixed}
Assume \(\DC_{\kappa^+}\). If \(\sS_0=\tup{\bbP_0,\sG_0,\sF_0}\) and \(\sS_1=\tup{\bbP_1,\sG_1,\sF_1}\) are symmetric systems, \(\kappa^+ \geq \aleph^\ast(\bbP_0 \cup \sF_0)\), \(\bbP_1\) is \(\kappa^+\)-closed and \(\sF_1\) is \(\kappa^+\)-complete then, setting \(\sS=\sS_0\times\sS_1\), we have \(\1\forces_\sS(\forall X)\aleph^\ast(X)>\check{\kappa}^+\implies\aleph(X)\geq\check{\kappa}^+\).
\end{prop}

\begin{proof}
Suppose that \(\tup{p_0,p_1}\forces_\sS``\ddf\) is a surjection \(\ddX\to\check{\kappa}^+\)''. Using \(\DC_{\kappa^+}\) and the \(\kappa^+\)-closure of \(\bbP_1\), construct a chain of conditions \(\tup{p_1^\alpha\mid\alpha<\kappa^+}\) below \(p_1\) such that, for all \(\alpha<\kappa^+\), there is \(p_0^\alpha\leq p_0\) and an \(\sS\)-name \(\ddx\) such that \(\tup{p_0^\alpha,p_1^\alpha}\forces_\sS\ddf(\ddx)=\check{\alpha}\). For \(q_0\leq p_0\), let \(\alpha_{q_0}=\sup\Set{\alpha<\kappa^+\mid(\exists\ddx\in V^\sS)\tup{q_0,p_1^\alpha}\forces_\sS\ddf(\ddx)=\check{\alpha}}\). Since \(\sup\Set{\alpha_{q_0}\mid q_0\leq p_0}=\kappa^+\) and \(\kappa^+\geq\aleph^\ast(\bbP_0)\), there must be a condition \(q_0\leq p_0\) such that \(\alpha_{q_0}=\kappa^+\). Using \(\DC_{\kappa^+}\), construct a sequence \(\tup{\ddx_\alpha\mid\alpha\in C}\) of \(\sS\)-names such that \(C\subseteq\kappa^+\) is cofinal and, for all \(\alpha\in C\), \(\tup{q_0,p_1^\alpha}\forces_\sS\ddf(\ddx_\alpha)=\check{\alpha}\). Since \(\kappa^+\geq\aleph^\ast(\sF_0)\), we also find that there is \(H\in\sF_0\) and a further cofinal subsequence \(D\subseteq C\) such that, for all \(\alpha\in D\), there is \(H_\alpha\in\sF\) such that \(\sym(\ddx_\alpha)=H\times H_\alpha\). \(\kappa^+\) is regular (by \(\DC_{\kappa^+}\)) and so we may enumerate the first \(\kappa\)-many elements of \(D\) as \(\Set{\gamma_\alpha\mid\alpha<\kappa}\). By the \(\kappa^+\)-completeness of \(\sF_1\) we have that \(\Set{\tup{\check{\alpha},\ddx_{\gamma_\alpha}}^\bullet\mid\alpha<\kappa}^\bullet\) is an \(\sS\)-name that \(\tup{q_0,p_0^\gamma}\) forces to be an injection \(\check{\kappa}\to\ddX\).
\end{proof}

Note that \cref{prop:preserve-oblateness-mixed} holds even in the case that \(\kappa\) is singular. Another way of stating the conclusion of \cref{prop:preserve-oblateness-mixed} is that, in the symmetric extension by \(\sS\), there cannot be \(X\) such that \(\aleph(X) < \kappa^+ < \aleph^\ast(X)\), so if \(\aleph(X) = \kappa < \aleph^\ast(X)\) then \(\aleph^\ast(X) = \kappa^+\). This is a familiar situation explaining the presence of `oblate cardinals' in models of \(\SVC\) (from \cite{ryan-smith_acwo_2024}).

\subsection{Local reflections of choice}\label{s:levys-axiom;ss:lrc}

In \cite{blass_injectivity_1979}, Blass introduced the axiom of \emph{small violations of choice}, stating that ``we call [it] the axiom of small violations of choice (\(\SVC\)) since it says that, in some sense, \emph{all the failure of choice occurs within a single set \(S\)} [emphasis added]'' \cite[p.~41]{blass_injectivity_1979}. As evidence, he provides the following theorem (attributing the idea behind the proof of \cref{item:bpi-lrc} to Pincus).

\begin{thm}[{Blass and Blass--Pincus}]
``Assume \(\SVC\) with \(S\).
\begin{enumerate}[label=\textup{(\alph*)}]
\item \(\AC\) holds if and only if there is a choice function for the nonempty subsets of \(S\).
\item\label{item:bpi-lrc} Let \(S^\lomega\) be the set of finite sequences of members of \(S\). The Boolean prime ideal theorem holds if and only if there is an ultrafilter \(U\) on \(S^\lomega\) that is regular in the sense that \(\Set{p \in S^\lomega \mid s \in \operatorname{range}(p)} \in U\) for every \(s \in S\).''
\end{enumerate}
\textup{\cite[Theorem~4.1, p.~41]{blass_injectivity_1979}}
\end{thm}

This concept holds for even more refined choice principles and was explored by Ryan-Smith in \cite{ryan-smith_local_2025}, calling the idea that choice principles reflect back to the seed \emph{local reflections of choice}. Continuing this theme, we provide the following reflection of extendable choice. Recall that \(\AL_\alpha(S)\) says that any approachable function \(F\colon\alpha\to\power(S)\) has a selector.

\begin{prop}\label{prop:al-lrc}
Assume \(\SVC^+(S)\). For all \(\alpha\), \(\AL_\alpha\) is equivalent to \(\AL_\alpha(S)\).
\end{prop}

\begin{proof}
Certainly \(\AL_\alpha\) implies \(\AL_\alpha(S)\), so assume \(\AL_\alpha(S)\) and let \(F\colon\alpha\to X\) be approachable, assuming without loss of generality that \(\bigcup X\subseteq S\times\eta\) for some ordinal \(\eta\). Define \(G\colon\alpha\to\power(S)\) by setting \(G(\beta)=\Set{s\in S\mid(\exists\delta<\eta)\tup{s,\delta}\in F(\beta)}\). We claim that \(G\) is approachable: firstly, since \(F\) is approachable, \(G(\beta)\neq\emptyset\) for all \(\beta<\alpha\); and secondly, if \(\gamma\mapsto\tup{s_\gamma,\delta_\gamma}\) is a selector for \(F\res\beta\), then \(\gamma\mapsto s_\gamma\) is a selector for \(G\res\beta\). By \(\AL_\alpha(S)\), there is a selector \(t\colon\alpha\to S\) for \(G\). Then
\begin{align*}
\alpha&\to\bigcup X\\
\beta&\mapsto\tup{t(\beta),\min\Set*{\delta<\eta\mid\tup{t(\beta),\delta}\in F(\beta)}}
\end{align*}
is a selector for \(F\).
\end{proof}

\begin{qn}\label{qn:improve-lrc}
Can \cref{prop:al-lrc} be improved to \(\SVC(S)\)? Any counter-example would require \(\cf(\alpha) > \omega\) because \(\AL_\omega\) is equivalent to \(\AC_\omega\) and (by \cite[Proposition~3.2]{ryan-smith_local_2025}) \(\SVC(S)\) implies that \(\AC_\omega\) is equivalent to \(\AC_\omega(S)\) (is equivalent to \(\AL_\omega(S)\)).
\end{qn}

\subsection{Large cardinals and \(\AL\)}\label{s:levys-axiom;ss:large-cardinals}

Let us remark that certain large cardinals impose restrictions on the preservation or violation of \(\AL\). The following definition of supercompactness (equivalent to the usual in \(\ZFC\)) is from Woodin \cite[Definition~220]{woodin_suitable_2010} and has been identified as the `correct' definition of supercompactness in \(\ZF\).

\begin{defn}
A cardinal \(\kappa\) is \emph{supercompact} if for all \(\alpha>\kappa\) there is \(\beta>\alpha\) and an elementary embedding \(j\colon V_\beta\to N\) with critical point \(\kappa\) such that \(j(\kappa)>\alpha\) and \(N^{V_\alpha}\subseteq N\).
\end{defn}

\begin{thm}\label{thm:supercompact-massive-failure}
If \(\kappa\) is supercompact and \(\AL_\delta\) fails for all \(\delta<\kappa\) then \(\AL_\eta\) fails for all infinite limit ordinals \(\eta\).
\end{thm}

\begin{proof}
Let \(\eta\) be an arbitrary infinite limit ordinal, \(\alpha>\eta\) be large enough that there is a witness for all of the failures of \(\AL_\delta\) for all \(\delta<\kappa\) in \(V_\alpha\), and \(j\colon V_\beta\to N\) be a supercompact embedding for some \(\beta>\alpha\). By elementarity, we have that \(N\models(\forall\delta<j(\kappa))\lnot\AL_\delta\) so, in particular, \(N\models\lnot\AL_\eta\) (given that \(j(\kappa)>\eta\)). Hence there is an \(\eta\)-approachable function \(F\) in \(N\) with no selector. However, if a selector for \(F\) existed in \(V\) then it would appear in \(N\) because \(N^{V_\alpha}\subseteq N\). Thus \(F\) witnesses the failure of \(\AL_\eta\) in \(V\).
\end{proof}

Note that the converse `if \(\AL_\delta\) holds for all \(\delta<\kappa\) then \(\AL_\delta\) holds for all \(\delta\)' is just a rephrasing of the fact that if \(\kappa\) is supercompact and \(\AC_{{<}\kappa}\) holds then \(\AC\) holds.

It is likely that other similar structural analyses can be made for large cardinals defined critically (per \cite{hayut_critical_2020}). Note that if the assumption in \cref{thm:supercompact-massive-failure} were to hold, then the universe would necessarily be a model of \(\lnot\SVC\). In particular, this would witness a nontrivial failure of choice above a supercompact.

\subsection{Approachable dependent choice}

Let us recall why \(\DC_{\omega_1}\) implies \(\DC\). Suppose that \(\DC\) fails, witnessed by the tree \(T\) with no maximal nodes or chains of order type \(\omega\). Then \(T\) is automatically \(\sigma\)-closed, for any countable chain must be finite and so certainly has an upper bound. However, there is not even any chain of order type \(\omega\) on \(T\), let alone \(\omega_1\), so indeed \(\DC_{\omega_1}\) fails. This falls into a similar niche regarding how \(\AC_{\omega_1}\) implies \(\AC_\omega\). Inspired by Levy's definition of \(\AL\), we present the following.

\begin{defn}
Let \(\delta\) be an ordinal. A \emph{\(\delta\)-approachable tree} is a \(\cf(\delta)\)-closed tree \(T\) of height at least \(\delta\) such that, for all \(\alpha < \beta < \delta\) and all \(s \in T_\alpha\) there is \(t \in T_\beta\) with \(s < t\). The \emph{principle of extendable dependent choice} (for \(\delta\)), \(\EDC_\delta\), is the statement that every \(\delta\)-approachable tree has a chain of order type \(\delta\).
\end{defn}

\(\EDC\) acts in many ways to \(\AL\) as \(\DC\) does to \(\AC\).

\begin{thm}
For all ordinals \(\alpha\) and \(\beta\), the following hold:
\begin{enumerate}
\item\label{item:adc-0} \(\EDC_0\);
\item\label{item:adc-alpha-plus-1} \(\EDC_{\alpha+1}\);
\item\label{item:adc-cf} if \(\cf(\alpha) = \cf(\beta)\) then \(\EDC_\alpha \lra \EDC_\beta\);
\item\label{item:adc-to-al} \(\EDC_\alpha \to \AL_\alpha\); and
\item\label{item:dc-to-adc} \(\DC_\alpha \to \EDC_\alpha\).
\end{enumerate}
\end{thm}

\begin{proof}
\cref{item:adc-0,item:adc-alpha-plus-1} are straightforward.

For \cref{item:adc-cf}, suppose that \(\EDC_\alpha\) holds and \(\cf(\alpha) = \cf(\beta)\). Let \(T\) be a \(\beta\)-approachable tree. If \(\alpha < \beta\), then let \(\tup{\delta_\gamma \mid \gamma < \alpha}\) be a strictly increasing cofinal sequence in \(\beta\) and let \(T' = \bigcup_{\gamma < \alpha} T_{\delta_\gamma}\). That is, we trim \(T\) to only the levels given by the \(\delta_\gamma\). Then \(T'\) is \(\alpha\)-approachable and so has a chain \(c'\) of order type \(\alpha\). Let \(c = \Set{s \in T \mid (\exists t \in c') s \leq t}\). Then \(c\) is a chain of order type \(\beta\) in \(T\). If instead \(\alpha \geq \beta\) then let \(\tup{\delta_\gamma \mid \gamma < \beta}\) be a strictly increasing cofinal sequence in \(\alpha\) with \(\delta_0 = 0\). Informally, we shall define \(T'\) by stretching \(T\) to be of height \(\alpha\), setting level \(\gamma\) of \(T\) to be level \(\delta_\gamma\) of \(T'\) and filling the gaps between levels \(\delta_\gamma\) and \(\delta_{\gamma+1}\) with additional copies of \(T_\gamma\). Formally, we define \(T'\) to be the set of functions \(f \colon \varep \to \bigcup T\) for \(\varep < \alpha\) such that:
\begin{enumerate}
\item for all \(\gamma < \beta\), \(f \res \Set{\zeta \mid \delta_\gamma \leq \zeta < \delta_{\gamma+1}}\) is constant;
\item if \(\delta_\gamma \leq \zeta < \delta_{\gamma+1}\) and \(\zeta < \varep\) then \(f(\zeta) \in T_\gamma\); and
\item \(f``\varep\) is a chain in \(T\).
\end{enumerate}
Then \(f \leq g\) if \(f \subseteq g\). This is an \(\alpha\)-approachable tree and so must have a chain \(c'\) of order type \(\alpha\). In this case, \(\Set{f``\dom(f) \mid f \in c'}\) is a chain in \(T\) of order type \(\beta\) as required.

To prove \cref{item:adc-to-al}, suppose that \(F\) is \(\alpha\)-approachable. Let \(T\) be the tree of partial selectors for \(F\), with \(f \leq g\) if \(f \subseteq g\). Then, given \(\beta < \gamma < \alpha\) and \(f \in T_\beta\), let \(g \in T_\gamma\) be arbitrary. Then \(f \cup (g \res \Set{ \delta \mid \beta \leq \delta < \gamma}) \in T_\gamma\) extends \(f\). Hence, by \(\EDC_\alpha\) there is a chain \(c\) of order type \(\alpha\). We may extract from \(c\) the selector \(f = \bigcup c\).

The proof of \cref{item:dc-to-adc} is also straightforward, as any \(\alpha\)-approachable tree must have no maximal nodes (when \(\alpha\) is a limit) and thus fall under the purview of \(\DC_\alpha\).
\end{proof}

We shall not spend any more time on extendible dependent choice, though we are interested in the possible spectra of extendible dependent choice.

\begin{qn}\label{qn:generous-adc}
Do \(\kappa\)-descending distributive symmetric systems preserve \(\EDC_\kappa\)?
\end{qn}

\section{The Hartogs--Lindenbaum spectrum}\label{s:hartlin}

This section regards possible configurations of the \emph{Hartogs--Lindenbaum spectrum}, from \cite{ryan-smith_acwo_2024}, defined for \(M\models\ZF\) as
\begin{equation*}
\Spec_\aleph(M)=\Set*{\tup{\aleph(X),\aleph^\ast(X)}\mid X\in M},
\end{equation*}
where \(\aleph(X)\) and \(\aleph^\ast(X)\) are calculated inside \(M\). In \cite{ryan-smith_acwo_2024}, Ryan-Smith demonstrates several constraints on this spectrum under the assumption of \(\SVC\), both as lower and upper bounds. However, this used delicate machinery arising from the (equivalent) definition of \(\SVC\) `there is an inner model of \(\ZFC\) of which the universe is a set-generic symmetric extension'. In \cref{s:hartlin;ss:going-up} we improve upon the lifting constructions of Ryan-Smith by omitting the assumption of \(\SVC\), creating more internal bounds on the constructions and giving explicit calculations for the Hartogs and Lindenbaum numbers of the new eccentric sets. In \cref{s:hartlin;ss:spectra} we improve upon the upper bounds of the spectrum under the assumption of \(\SVC\), using Peng's well-ordered productivity of Lindenbaum numbers and a simplified presentation of \(\SVC\) to give a tighter bound on the spectrum of symmetric extensions.

\subsection{Eccentric arithmetic}\label{s:hartlin;ss:eccentric-arithmetic}

The arithmetic of Hartogs and Lindenbaum numbers was touched on in the introduction, but we will require a few more specific details for \cref{thm:going-up-singular-or-limit}. Recall that for sets \(A\) and \(B\), we denote by \(\Inj{A}{B}\) the set of injections \(A \to B\).

\begin{prop}\label{prop:inj-vs-power-cardinality}
If \(X\neq\emptyset\) is Dedekind-finite, then for all \(1<n<\omega\), \(\abs{\Inj{n}{X}}<\abs{X^n}\). If \(X\) is Dedekind-infinite, then for all \(\alpha<\aleph(X)\), \(\abs{\Inj{\alpha}{X}}=\abs{X^\alpha}\).
\end{prop}

\begin{proof}
Firstly, let \(1<n<\omega\). We shall show that if \(\abs{\Inj{n}{X}}=\abs{X^n}\) then \(X\) is Dedekind-infinite. Note that for finite \(X\) the result is immediate, so assume that \(X\) is infinite and let \(b_0,\dots,b_{n-1}\in X\) be an arbitrary sequence of distinct elements of \(X\). As we are assuming \(\abs{\Inj{n}{X}}=\abs{X^n}\), we may obtain an injection \(F\colon X^n\to\Inj{n}{X}\). We shall definably extend \(b_0,\dots,b_{n-1}\) to a sequence of \(\omega\)-many distinct elements, concluding that \(X\) is Dedekind-infinite. Suppose that we have already defined \(b_0,\dots,b_{m-1}\) distinct some \(m\geq n\). As \(\omega^\lomega=\omega\), we may always uniformly well-order all finite powers of finite indexed sets. Since \(\abs{\Inj{n}{\Set{b_0,\dots,b_{m-1}}}}=n!\genfrac{(}{)}{0ex}{}{m}{n}<m^n\) (true for \(n\geq2\)) so for some \(\tup{i_0,\dots,i_{n-1}}\in m^n\), \(F(\tup{b_{i_j}\mid j<n})``n\nsubseteq\Set{b_0,\dots,b_{m-1}}\). So take \(\tup{i_0,\dots,i_{n-1}}\) to be least (in our canonical linear ordering) such that
\begin{equation*}
F(\tup{b_{i_j}\mid j<n})``n\cap\Set{b_0,\dots,b_{m-1}}\neq\emptyset
\end{equation*}
and define \(b_m=F(\tup{b_{i_j}\mid j<n})(k)\), where \(k\) is taken to be the least such that \(F(\tup{b_{i_j}\mid j<n})(k)\notin\Set{b_0,\dots,b_{m-1}}\). Thus we have inductively defined a sequence of distinct elements of \(X\), so it is Dedekind-infinite as required.

Now let \(X\) be Dedekind-infinite and \(\alpha<\aleph(X)\). Then \(\alpha\times\alpha\times\omega<\aleph(X)\), so let \(f\colon\alpha\times\alpha\times\omega\to X\) be an injection. We shall define a function \(F\colon X^\alpha\to\Inj{\alpha}{X}\) and show that it is injective. For \(g\colon\alpha\to X\) and \(\beta<\alpha\), we set
\begin{equation*}
F(g)(\beta)=\begin{cases}
g(\beta)&g(\beta)\notin g``\beta\cup f``(\alpha\times\alpha\times\omega)\\
f(\gamma,\delta,n+1)&g(\beta)=f(\gamma,\delta,n)\notin g``\beta\\
f(\gamma,\beta,0)&\gamma<\beta\text{ is least such that }g(\beta)=g(\gamma).
\end{cases}
\end{equation*}
This is perhaps better illuminated by an algorithmic description: First reserve the set \(A=f``(\alpha\times\alpha\times\omega)\). Given \(\beta\), if \(g(\beta)\) is distinct from all prior \(g(\alpha)\) and \(g(\beta)\notin A\), then we set \(F(g)(\beta)=g(\beta)\). If instead, \(g(\beta)=g(\gamma)\), where \(\gamma<\beta\) is the earliest that \(g(\beta)\) appears, then we set \(F(g)(\beta)=f(\gamma,\beta,0)\). Now, for any \(g(\beta)\in A\), we increase the \(\omega\) co-ordinate by one, so if \(g(\beta)=f(\gamma,\delta,n)\) then we set \(F(g)(\beta)=f(\gamma,\delta,n+1)\).

In this way, given \(F(g)\), one can recover \(g\): if \(F(g)(\beta)\notin A\) then \(g(\beta)\) is \(F(g)(\beta)\); if \(F(g)(\beta)=f(\gamma,\beta,0)\), then \(F(g)(\beta)=g(\gamma)=F(g)(\gamma)\); and if \(F(g)(\beta)=f(\gamma,\delta,n+1)\) then \(g(\beta)=f(\gamma,\delta,n)\). Thus, \(F\) is an injection as required.
\end{proof}

\begin{prop}\label{prop:power-equivalence-hartogs}
For all infinite \(X\) and \(0<\alpha<\aleph(X)\), \(\aleph(\Inj{\alpha}{X})=\aleph(X^\alpha)\) and \(\aleph^\ast(\Inj{\alpha}{X})=\aleph^\ast(X^\alpha)\).
\end{prop}

\begin{proof}
For Dedekind-infinite \(X\) \cref{prop:inj-vs-power-cardinality} gives us that \(\abs{\Inj{\alpha}{X}}=\abs{X^\alpha}\) and the result follows, so instead assume that \(X\) is Dedekind-finite and let \(n < \aleph_0 = \aleph(X)\). Then \(\aleph(X^n)=\aleph(X)^n=\aleph(X)\). On the other hand, \(0<n\), so \(\abs{X}\leq\abs{\Inj{n}{X}}\) and thus \(\aleph(X)\leq\aleph(\Inj{n}{X})\leq\aleph(X^n)=\aleph(X)\). To establish the result for \(\aleph^\ast\) we require an alternative strategy, as Lindenbaum numbers are not productive.

For the rest of the proof, let \(\Surj{A}{B}\) denote the set of surjections \(A\to B\). Given any function \(f\colon n\to X\) we may characterise \(f\) via an injection \(k_f\colon m\to X\), where \(m=\abs{f``n}\leq n\), and a surjection \(e_f\colon n\to m\), where \(k_f(i)\) is the \(i\)\textup{th} unique element of the sequence \(\tup{f(j)\mid j<n}\) and \(e_f\) is the unique function such that \(f(j)=k_f(e_f(j))\) for all \(j<n\). Since we can recover \(f\) from this data, \(f\mapsto\tup{k_f,e_f}\) provides a bijection between \(X^n\) and \(\bigcup_{1 \leq m \leq n}\Inj{m}{X}\times\Surj{n}{m}\). Therefore, noting that \(\Surj{n}{m}\) is finite and that, if \(X\) is infinite and \(F \neq \emptyset\) is finite then \(\aleph^\ast(X \times F) = \aleph^\ast(X)\),
\begin{align*}
\aleph^\ast(X^n)&=\aleph^\ast\parenth{\bigcup_{m=1}^n\parenth{\Inj{m}{X}\times\Surj{n}{m}}}\\
&=\sum_{m=1}^n\aleph^\ast\parenth{\Inj{m}{X}\times\Surj{n}{m}}\\
&=\sum_{m=1}^n\aleph^\ast\parenth{\Inj{m}{X}}\\
&\leq n\times\aleph^\ast\parenth{\Inj{n}{X}}\\
&=\aleph^\ast\parenth{\Inj{n}{X}}.
\end{align*}
We certainly have that \(\abs{\Inj{n}{X}}\leq\abs{X^n}\), so we conclude that \(\aleph^\ast(X^n)=\aleph^\ast(\Inj{n}{X})\).
\end{proof}

\begin{prop}\label{prop:sigma-lambda}
For all \(X\),
\begin{equation*}
\sup\Set*{\aleph(X^\alpha)\mid\alpha<\aleph(X)}=\sup\Set*{\aleph(\alpha^\alpha)\mid\alpha<\aleph(X)}.
\end{equation*}
\end{prop}

\begin{proof}
Let \(\mu=\aleph(X)\). If \(\mu=\aleph_0\) then \(\aleph(X^\alpha)=\aleph_0\) for all \(\alpha<\aleph_0\) (by the productivity of Hartogs numbers), so \(\sup\Set{\aleph(X^\alpha)\mid\alpha<\aleph_0}=\sup\Set{\aleph(\alpha^\alpha)\mid\alpha<\aleph_0}=\aleph_0\), so let us assume instead that \(X\) is Dedekind-infinite.

(\(\geq\)).\quad{}If \(\abs{\alpha}\leq\abs{X}\) then \(\abs{\alpha^\alpha}\leq\abs{X^\alpha}\), so certainly \(\aleph(\alpha^\alpha)\leq\aleph(X^\alpha)\) for all \(\alpha<\mu\).

(\(\leq\)).\quad{}We shall show that for all infinite \(\alpha<\mu\),
\begin{equation*}
\aleph(X^\alpha)\leq\sup\Set{\aleph(\kappa^\kappa)\mid\alpha\leq\kappa<\mu}.
\end{equation*}
Since \(X\) is Dedekind-infinite, \(\abs{X^\alpha}=\abs{\Inj{\alpha}{X}}\) (\cref{prop:power-equivalence-hartogs}), so let \(f\colon\lambda\to\Inj{\alpha}{X}\) be an injection for some \(\lambda<\aleph(X^\alpha)\). By lexicographic ordering, \(\tup{f(\beta,\gamma)\mid\beta<\lambda,\gamma<\alpha}\) is a well-ordered subset of \(X\) with order type \(\kappa\) for some \(\alpha\leq\kappa<\mu\). By this identification, \(f\) induces an injection \(\hat{f}\colon\lambda\to\Inj{\alpha}{\kappa}\) and thus (since \(\kappa\) is Dedekind-infinite) \(\lambda<\aleph(\kappa^\alpha)\). That is, if \(\lambda<\aleph(X^\alpha)\) then there exists \(\kappa<\mu\) such that \(\lambda<\aleph(\kappa^\alpha)\), so \(\aleph(X^\alpha)\leq\sup\Set{\aleph(\kappa^\alpha)\mid\alpha\leq\kappa<\mu}\). It follows that
\begin{align*}
\sup\Set{\aleph(X^\alpha)\mid\alpha<\mu}&\leq\sup\Set{\aleph(\kappa^\alpha)\mid\alpha\leq\kappa<\mu}\\
&=\sup\Set{\aleph(\kappa^\kappa)\mid\kappa<\mu}.\qedhere
\end{align*}
\end{proof}

\subsection{Lifting eccentric sets}\label{s:hartlin;ss:going-up}

\cref{thm:going-up-regular,thm:going-up-singular-or-limit} provide two flavours of `lifting' witnesses of the failure of \(\AL\). In the first case, if \(\aleph(X) \leq \mu < \aleph^\ast(X)\) then whenever \(\cf(\lambda) = \mu\) we may produce a set \(Y\) such that \(\aleph(Y) = \lambda\) and \(\aleph^\ast(Y) = \aleph^\ast(X) \times \lambda^+\). This is a very powerful lifting tool as it directly gives us the sets \(Y\) for all \(\lambda \geq \mu\) with \(\cf(\lambda) = \mu\), rather than needing to appeal to the phrase `for large enough \(\lambda\).' However, this power comes from the requirement that \(\mu\) is regular. The second, \cref{thm:going-up-singular-or-limit}, instead allows for \(\aleph(X) = \mu\) to be singular or a limit, with no requirements on \(\aleph^\ast(X)\) (even if \(X\) itself is not eccentric, having \(\aleph(X)\) be singular or a limit is a violation of \(\AC_\WO\) and thus \(\AL\), see \cite[Theorem~3.4]{ryan-smith_acwo_2024}). In this case, however, it is unknown if for all \(\lambda \geq \mu\) with \(\cf(\lambda) = \mu\) there is \(Y\) such that \(\aleph(Y) = \lambda < \aleph^\ast(Y)\). Instead our construction requires that \(\lambda \geq \sup \Set{ \aleph(\alpha^\alpha) \mid \alpha < \mu}\). \cref{thm:going-up-regular,thm:going-up-singular-or-limit} correspond very closely to \cite[Lemmas~4.19 and 4.18]{ryan-smith_acwo_2024} respectively and have very similar proofs. However, in both cases we have removed the requirement of \(\SVC\), weakened the assumptions and given a precise calculation of the Hartogs and Lindenbaum numbers constructed. As such, we still present the proofs in their new, clearer forms.

\begin{thm}\label{thm:going-up-regular}
Let \(X\) be a set and \(\lambda\) a cardinal such that \(\aleph(X)\leq\cf(\lambda)<\aleph^\ast(X)\). Then there is a set \(Y\) such that \(\aleph(Y)=\lambda\) and \(\aleph^\ast(Y)=\aleph^\ast(X)\times\lambda^+\).
\end{thm}

\begin{proof}
Let \(\mu=\cf(\lambda)\) and let \(\tup{\lambda_\alpha\mid\alpha<\mu}\) be a strictly increasing cofinal sequence in \(\lambda\). Since \(\mu<\aleph^\ast(X)\), we may partition \(X\) as \(\bigcup_{\alpha<\mu}X_\alpha\), with each \(X_\alpha\) non-empty. Let \(Y=\bigcup_{\alpha<\mu}X_\alpha\times\lambda_\alpha\). Note that we immediately have \(\lambda\leq\aleph(Y)\) since, for each \(\alpha\), we may pick \(x\in X_\alpha\) and consider the injection \(\lambda_\alpha\to Y\) given by \(\gamma\mapsto\tup{x,\gamma}\). Furthermore, projection to the first and second co-ordinates respectively beget surjections \(Y\to X\) and \(Y\to\lambda\), so \(\aleph^\ast(Y)\geq\max\Set{\aleph^\ast(X),\lambda^+}\). Note also that \(Y\subseteq X\times\lambda\), so \(\aleph^\ast(Y)\leq\aleph^\ast(X\times\lambda)=\aleph^\ast(X)\times\lambda^+\). Thus, \(\aleph^\ast(Y)=\aleph^\ast(X)\times\lambda^+\) as required.

Finally, it remains to show that \(\aleph(Y)\ngtr\lambda\). Assume for a contradiction that \(\tup{\tup{x_\gamma,\delta_\gamma}\mid\gamma<\lambda}\) is an enumeration of \(\lambda\)-many distinct elements of \(Y\) and, for each \(\gamma<\lambda\), let \(\tup{x_\gamma,\delta_\gamma}\in X_{\alpha_\gamma}\times\lambda_{\alpha_\gamma}\) for some \(\alpha_\gamma<\mu\). Let \(\tau\) be the order type of \(\Set{\alpha_\gamma\mid\gamma<\lambda}\). Then we can extract a function \(\tau\to X\) by sending \(\varep\) to \(x_\gamma\), where \(\gamma\) is least such that \(\alpha_\gamma\) is the \(\varep\)\textup{th} element of \(\Set{\alpha_\gamma\mid\gamma<\lambda}\). Since the \(X_\alpha\) are pairwise disjoint, this is an injection, so \(\tau<\mu\). Since \(\mu\) is regular, this means that there is \(\beta<\mu\) such that \(\Set{\tup{x_\gamma,\delta_\gamma}\mid\gamma<\lambda}\subseteq T\defeq\bigcup_{\alpha<\beta}X_\alpha\times\lambda_\alpha\). However, \(T\subseteq X\times\lambda_\beta\) and so \(\aleph(T)\leq\aleph(X\times\lambda_\beta)=\aleph(X)\times\lambda_\beta^+\leq\lambda\), contradicting that \(\gamma\mapsto\tup{x_\gamma,\delta_\gamma}\) is an injection \(\lambda\to T\).
\end{proof}

As mentioned in the introduction, this provides a positive answer to a question posed by Ryan-Smith: ``Let \(\mu\) be weakly inaccessible and suppose that for some set \(X\), \(\aleph(X)<\mu<\aleph^\ast(X)\). Must there exist \(Y\) such that \(\aleph(Y)=\mu\)?'' \cite[Question~5.3, p.~222]{ryan-smith_acwo_2024}

Furthermore, we cannot strengthen \cref{thm:going-up-regular} by merely requiring that there is \(\mu\) such that \(\aleph(X)\leq\mu<\aleph^\ast(X)\) and \(\cf(\lambda)=\cf(\mu)\), rather than \(\mu\) being regular. For example, starting in \(V\models\ZFC\), let \(M\) be the symmetric extension from \cite[Theorem~3.1]{karagila_which_2024} that adds a set \(X\) with \(\aleph(X)=\aleph_1\) and \(\aleph^\ast(X)=\aleph_{\omega+1}\). As noted in \cite{karagila_which_2024}, \(M\models\DC\) in this case. Hence there is no \(Y\) such that \(\aleph(Y)=\aleph_\omega\), even though \(\aleph(X)<\aleph_\omega<\aleph^\ast(X)\).

The following theorem improves upon \cite[Lemma~4.18]{ryan-smith_acwo_2024} again by removing requirements surrounding \(\SVC\), simplifying the language and allowing the case that \(\aleph(X)\) is a singular successor. In fact, if \(\kappa^+\) is singular then \(X = \kappa\) is a valid input for \cref{thm:going-up-singular-or-limit}.

\begin{thm}\label{thm:going-up-singular-or-limit}
Let \(\aleph(X)=\mu\) be singular or a limit and let \(\lambda\geq\sup\Set{\aleph(\alpha^\alpha)\mid\alpha<\mu}\). Then there is a set \(Y\) such that \(\aleph(Y)=\lambda\) and \(\aleph^\ast(Y)=\aleph^\ast(\Inj{{<}\mu}{X})\times\lambda^+\).
\end{thm}

\begin{proof}
Let \(\tup{\delta_\alpha\mid\alpha<\mu}\) be a cofinal sequence in \(\lambda\) with \(\delta_0>0\). For each \(\alpha<\mu\), let \(Y_\alpha=\Inj{\alpha}{X}\times\delta_\alpha\) and \(Y=\bigcup_{\alpha<\mu}Y_\alpha\). Since \(\aleph(X)=\mu\), for all \(\alpha<\mu\), \(\Inj{\alpha}{X}\neq\emptyset\), so \(\delta_\alpha\) embeds into \(Y_\alpha\subseteq Y\). Hence, \(\aleph(Y)\geq\lambda\). Projection to the second co-ordinate begets a surjection \(Y\to\lambda\) and hence \(\aleph^\ast(Y)>\lambda\).

Let \(\beta<\mu\). We shall show that \(\abs{\bigcup_{\alpha<\beta}Y_\alpha}\leq\abs{Y_\beta\times\beta}\):

We start by defining an injection \(F_0\colon\bigcup_{\alpha<\abs{\beta}}Y_\alpha\to Y_\beta\times\abs{\beta}\). Let \(c\in\Inj{\beta}{X}\). Then, for all \(\alpha<\abs{\beta}\), we have a definable embedding \(e_\alpha\colon\Inj{\alpha}{X}\to\Inj{\beta}{X}\) where \(e_\alpha(i)\) is given by concatenating the injections \(i\) and \(c\), minus any duplicates. Since \(\abs{\alpha}<\abs{\beta}\), the sequence \(i``\alpha\concat c``\beta\) (the concatenation of \(i``\alpha\) and \(c``\beta\)) will still have order type at least \(\beta\) after removing duplicates. \(e_\alpha(i)\res\alpha=i\), so \(e_\alpha\) is indeed an injection. Hence we obtain
\begin{equation*}
F_0\colon\bigcup\Set*{Y_\alpha\mid\abs{\alpha}<\abs{\beta}}\to\Inj{\beta}{X}\times\delta_\beta\times\abs{\beta}
\end{equation*}
given by, for \(\tup{f,\gamma}\in Y_\alpha\), \(F_0(f,\gamma)=\tup{e_\alpha(f),\gamma,\alpha}\).

\begin{claim}\label{claim:spectrum-of-bijections}
There is a sequence \(\tup{k_\alpha\mid \alpha<\beta,\,\abs{\alpha}=\abs{\beta}}\) such that, for all \(\alpha\), \(k_\alpha\) is a bijection \(\beta\to\alpha\).
\end{claim}

\begin{poc}
Let \(\eta=\abs{\beta}\) and \(f\colon\beta\to\eta\) a bijection. For \(\eta \leq \alpha < \beta\), let \(z_\alpha=f``\alpha\). Since \(f\) is a bijection and \(\abs{\alpha}=\eta\), \(\abs{z_\alpha}=\eta\) as well. \(z_\alpha\subseteq\eta\), so \(\ot(z_\alpha)\leq\eta\). Therefore, given that \(\eta\) is a cardinal, \(\ot(z_\alpha)=\eta\), with increasing enumeration \(\tup{\beta_\delta\mid\delta<\eta}\). Let \(k_\alpha\colon\beta\to\alpha\) be given by \(k_\alpha(\delta)=f^{-1}(\beta_\delta)\).
\end{poc}

Let \(\tup{k_\alpha\mid\abs{\alpha}=\abs{\beta}}\) be as in \cref{claim:spectrum-of-bijections}. Then we have an injection
\begin{equation*}
F_1\colon \bigcup\Set*{Y_\alpha\mid\abs{\beta}\leq\alpha<\beta}\to\Inj{\beta}{X}\times\delta_\beta\times(\beta\setminus\abs{\beta})
\end{equation*}
given by, for \(\tup{f,\gamma}\in Y_\alpha\), \(F_1(f,\gamma)=\tup{fk_\alpha,\gamma,\alpha}\).

Hence, \(\abs{\bigcup_{\alpha<\beta}Y_\alpha}\leq\abs{Y_\beta\times\beta}\) as required, so
\begin{align*}
\aleph\parenth{\bigcup\Set*{Y_\alpha\mid\alpha<\beta}}&\leq\aleph(Y_\beta\times\beta)\\
&=\aleph(X^\beta)\times\abs{\delta_\beta}^+\times\abs{\beta}^+&\\
&\leq\lambda,
\end{align*}
recalling from \cref{prop:sigma-lambda} that
\begin{align*}
\aleph(X^\beta)&\leq\sup\Set{\aleph(X^\alpha)\mid\alpha<\mu}\\
&=\sup\Set{\aleph(\alpha^\alpha)\mid\alpha<\mu}\\
&\leq\lambda.
\end{align*}
Hence, if \(f\colon\lambda\to Y\) is an injection, then \(f``\lambda\) intersects \(Y_\alpha\) for unboundedly many \(\alpha\). Given such a well-ordered sequence, if \(\mu\) is a limit then one can reconstruct an injection \(\mu\to X\), contradicting that \(\aleph(X)=\mu\). The case that \(\mu\) is a singular successor is somewhat more complicated.

\begin{claim}
If \(\mu=\chi^+\) is singular, then there is no injection \(\lambda\to Y\).
\end{claim}

\begin{poc}
By our prior work, any such injection would induce an injection \(h \colon \calC \to \Inj{{<}\mu}{X}\), where \(\calC\subseteq\mu\) is unbounded, and for each \(\alpha\in\calC\), \(h(\alpha)\) is an injection \(\alpha\to X\). By passing to a cofinal subsequence, we may assume that \(\abs{\calC}<\mu\). Consider \(I=\Set{h(\alpha)(\beta)\mid\beta<\alpha,\,\alpha\in\calC}\). By lexicographic ordering, \(I\) is well-orderable and thus \(\abs{I}=\chi\), say \(I=\Set{x_\gamma\mid\gamma<\chi}\). Hence we have a uniform sequence of injections \(i_\alpha\colon\alpha\to\chi\) for \(\alpha\in\calC\) where \(i_\alpha(\beta)=\gamma\) if \(h(\alpha)(\beta)=x_\gamma\). By taking a Mostowski collapse, we obtain bijections \(k_\alpha\colon\alpha\to\chi\) for all \(\alpha\in\calC\) such that \(\abs{\alpha}=\abs{\eta}\). Hence, we obtain a surjection \(g\colon\eta\times\calC\to\mu\) by \(g(\beta,\alpha)=k_\alpha^{-1}(\beta)\), contradicting that \(\mu>\abs{\eta\times\calC}=\eta\).
\end{poc}

Therefore, \(\aleph(Y)=\lambda\) as required. Note that
\begin{equation*}
\Inj{{<}\mu}{X}\times\Set{0}\subseteq Y\subseteq\Inj{{<}\mu}{X}\times\lambda,
\end{equation*}
so, setting \(\kappa=\aleph^\ast(\Inj{{<}\mu}{X})\),
\begin{align*}
\kappa&\leq\aleph^\ast(Y)\\
&\leq\aleph^\ast(\Inj{{<}\mu}{X}\times\lambda)\\
&=\kappa\times\lambda^+
\end{align*}
However, we showed earlier that \(\lambda<\aleph^\ast(Y)\), so \(\aleph^\ast(Y)=\kappa\times\lambda^+\).
\end{proof}

\subsection{Small violations of choice}\label{s:hartlin;ss:spectra}

Assuming \(\SVC(S)\), given a set \(X\) there is a least \(\eta\) such that \(\abs{X}\leq^\ast\abs{S\times\eta}\), witnessed by a surjection \(f\colon S\times\eta\to X\), say. Since we may assume that \(f\) is a \emph{partial} surjection without issue, we may assume without loss of generality that for all \(s\in S\), \(f\res\Set{s}\times\eta\) is an injection \(\dom(f)\cap\Set{s}\times\eta\to X\) and that \(\dom(f)\cap\Set{s}\times\eta=\Set{s}\times\alpha_s\) for some \(\alpha_s\leq\eta\). Furthermore, we may assume that for all \(x\in X\), there is a unique \(\alpha<\eta\) such that, for some \(s\in S\), \(f(s,\alpha)=x\).

\begin{defn}
A partial function \(f\colon S\times\eta\to X\), where \(\eta\) is an ordinal, is a \emph{skew-split surjection} if:
\begin{enumerate}
\item For each \(s\in S\) there is \(\alpha_s\leq\eta\) such that \(\dom(f)=\Set{\tup{s,\beta}\mid s\in S,\beta<\alpha_s}\);
\item for each \(s\in S\), \(f\res\Set{s}\times\alpha_s\) is an injection;
\item for each \(x\in X\) there is a unique \(\alpha_x<\eta\) such that, for some \(s\in S\), \(f(s,\alpha_x)=x\) (in particular, \(f\) is a surjection); and
\item \(\eta\) is the least ordinal such that \(\abs{X}\leq^\ast\abs{S\times\eta}\).
\end{enumerate}
By our arguments above, if \(\eta\) is least such that \(\abs{X}\leq^\ast\abs{S\times\eta}\) then there is a skew-split surjection \(S\times\eta\to X\).
\end{defn}

\begin{thm}\label{thm:product-restriction}
Let \(S\) and \(X\) be sets and \(\eta \in \Ord\) be such that \(\abs{X} \leq^\ast \abs{S \times \eta}\), with \(\eta\) least such that \(\abs{X} \leq^\ast \abs{S \times \eta}\). Then, using \(A\to B\) to mean \(A\leq B\), we have
\begin{equation*}
\begin{tikzcd}
\eta^+\ar[r]&\aleph^\ast(X)\ar[r]&\aleph^\ast(S)\times\eta^+\\
\eta\ar[u]\ar[r]&\aleph(X)\ar[u].
\end{tikzcd}
\end{equation*}
In particular, if \(\aleph^\ast(X)>\aleph^\ast(S)\) then \(\eta \leq \aleph(X) \leq \eta^+ = \aleph^\ast(X)\).

Furthermore, if \(\aleph(X)=\eta\) then \(\cf(\eta)<\aleph^\ast(S)\).
\end{thm}

\begin{proof}
Let \(f\colon S\times\eta\to X\) be skew-split. For each \(s\in S\), let \(\alpha_s\leq\eta\) be such that \(\dom(f)=\Set{\tup{s,\beta}\mid s\in S,\beta<\alpha_s}\) and, for each \(x\in X\), let \(\alpha_x\) be the unique ordinal such that, for some \(s\in S\), \(f(s,\alpha_x)=x\). Since \(\eta\) is minimal, we must have that \(\sup\Set{\alpha_s\mid s\in S}=\eta\) and so the function \(X\to\eta\) given by \(x\mapsto\alpha_x\) is a well-defined surjection. Thus, \(\aleph^\ast(X)\geq\eta^+\). Furthermore, for all \(s\in S\), \(f\res\Set{s}\times\alpha_s\) is an injection \(\alpha_s\to X\), so \(\aleph(X)\geq\eta\). Finally, we have that \(\abs{X}\leq^\ast\abs{S\times\eta}\), so \(\aleph^\ast(X)\leq\aleph^\ast(S\times\eta)=\aleph^\ast(S)\times\eta^+\) by \cref{thm:lindenbaum-wo-productivity}.

In the case that \(\aleph^\ast(X)>\aleph^\ast(S)\) we must have \(\aleph^\ast(S)\times\eta^+>\aleph^\ast(S)\), so \(\eta^+>\aleph^\ast(S)\). Therefore, we deduce
\begin{equation*}
\begin{tikzcd}
\eta^+\ar[r]&\aleph^\ast(X)\ar[r]&\eta^+\\
\eta\ar[r]\ar[u]&\aleph(X)\ar[u]
\end{tikzcd}
\end{equation*}
and the result follows.

Finally, if \(\aleph(X)=\eta\) then \(\alpha_s<\eta\) for all \(s\in S\). Hence, \(\eta=\sup\Set{\alpha_s\mid s\in S}\). Let \(\tup{\beta_\gamma\mid\gamma<\tau}\) be the unique increasing enumeration of \(\Set{\alpha_s\mid s\in S}\). Then \(\tau<\aleph^\ast(S)\) and \(\eta=\sup\Set{\beta_\gamma\mid\gamma<\tau}\), so \(\cf(\eta)\leq\tau<\aleph^\ast(S)\) as required.
\end{proof}

\begin{cor}\label{cor:svc-bounds-on-eccentricity}
Assume \(\SVC(S)\). Let \(X\) and \(\lambda\) be such that \(\aleph(X)\leq\lambda<\aleph^\ast(X)\). Either:
\begin{enumerate}
\item \(\aleph^\ast(X)\leq\aleph^\ast(S)\); or
\item \(\lambda=\aleph(X)<\aleph^\ast(X)=\lambda^+\) and \(\cf(\lambda)<\aleph^\ast(S)\).
\end{enumerate}
\end{cor}

\begin{prop}\label{prop:svc-singular-successor}
Assume \(\SVC(S)\). For all cardinals \(\lambda\), if \(\lambda^+\) is singular then \(\lambda^+<\aleph^\ast(S)\). In particular, if \(\aleph^\ast(S)\) is a successor then it is regular.
\end{prop}

\begin{proof}
Assume that \(\lambda^+\) is singular, with cofinal sequence \(\tup{\lambda_\alpha\mid\alpha<\mu}\), where \(\mu=\cf(\lambda^+)\) and \(\lambda_\alpha\geq\lambda\) for all \(\alpha<\mu\). Let \(I\) be the set of bijections \(\lambda\to\lambda_\alpha\) for some \(\alpha<\mu\) and let \(g\colon S\times\eta\to I\) be a surjection for some \(\eta\). We define
\begin{align*}
f\colon S\times\lambda\times\mu&\to\lambda^+\\
\tup{s,\beta,\alpha}&\mapsto\begin{cases}
g(s,\gamma)(\beta)&\text{$\gamma$ least such that $g(s,\gamma)\colon\lambda\to\lambda_\alpha$, if defined}\\
0&\text{otherwise.}
\end{cases}
\end{align*}
\(f\) is a surjection, so \(\aleph^\ast(S\times\lambda\times\mu)=\aleph^\ast(S)\times\lambda^+>\lambda^+\) and, thus, \(\lambda^+<\aleph^\ast(S)\) as required.
\end{proof}

\begin{cor}\label{cor:reg-suc}
Assume \(\SVC(S)\). For all \(X\), if \(\aleph^\ast(X)>\aleph^\ast(S)\) then \(\aleph^\ast(X)\) is a regular successor.
\end{cor}

\begin{proof}
If \(\aleph^\ast(X)>\aleph^\ast(S)\) then from \cref{thm:product-restriction} we get that \(\aleph^\ast(X)=\eta^+\), where \(\eta\) is least such that \(\abs{X}\leq^\ast\abs{S\times\eta}\), so certainly \(\aleph^\ast(X)\) is a successor. Given that \(\eta^+=\aleph^\ast(X)>\aleph^\ast(S)\), \(\eta^+\) must be regular by \cref{prop:svc-singular-successor}.
\end{proof}

\begin{cor}\label{cor:svc-seed-lindenbaum-regular-iff-successor}
Assume \(\SVC(S)\). Then \(\aleph^\ast(S)\) is regular if and only if it is a successor.
\end{cor}

\begin{proof}
By \cref{prop:svc-singular-successor} if \(\aleph^\ast(S)\) is a successor then it is regular. Suppose that \(\aleph(S) \leq \aleph^\ast(S) = \lambda\) is a regular limit. By \cref{thm:going-up-regular} there is \(X\) such that \(\aleph(X) = \lambda\) and \(\aleph^\ast(X) = \lambda^+\). However, by \cref{cor:svc-bounds-on-eccentricity} \(\cf(\aleph(X)) = \lambda < \aleph^\ast(S) = \lambda\), a contradiction. Hence if \(\aleph^\ast(S)\) is a limit then it is singular.
\end{proof}

While the fact that \(\aleph^\ast(X)\) is a successor was is already known for large \(\aleph^\ast(X)\) \cite[Lemma~4.14]{ryan-smith_acwo_2024}, the fact that \(\aleph^\ast(X)\) is necessarily regular is new. \cref{cor:svc-seed-lindenbaum-regular-iff-successor} gives us exactly two possibilities for \(\aleph^\ast(S)\): either it is a regular successor or it is a singular limit. We know that the first class is possible, witnessed by Cohen's first model that satisfies \(\SVC^+([A]^\lomega)\) and \(\aleph^\ast([A]^\lomega) = \aleph_1\), but we are unaware of any models of the second class.

\begin{qn}\label{qn:is-aleph-ast-s-regular}
Does \(\SVC(S)\) imply that \(\aleph^\ast(S)\) is regular?
\end{qn}

Combining \cref{cor:svc-bounds-on-eccentricity} with \cref{thm:going-up-regular,thm:going-up-singular-or-limit}, we obtain \cref{thm:svc-spectrum}.

\begin{thm}\label{thm:svc-spectrum}
Let \(M\models\SVC(S)\). Then there is a set \(C\) of regular cardinals less than \(\aleph^\ast(S)\) and a cardinal \(\Omega\) such that:
\begin{enumerate}
\item For all \(X\), if \(\aleph(X)<\aleph^\ast(X)\) then \(\cf(\aleph(X))\in C\);
\item for all \(X\), if \(\aleph(X)^+<\aleph^\ast(X)\) then \(\aleph^\ast(X)\leq\aleph^\ast(S)\); and
\item for all \(\lambda\geq\Omega\), if \(\cf(\lambda)\in C\) then there is \(X\) such that \(\aleph(X)=\lambda\) and \(\aleph^\ast(X)=\lambda^+\).
\end{enumerate}
\end{thm}

\section{The future}\label{s:future}

Let us re-iterate the open questions posed throughout this text. In some cases, additional exposition or context from the referenced question have been removed.

\begin{enumerate}
\item Is every \(\kappa\)-descending distributive forcing also strongly \(\kappa\)-descending distributive? (\cref{qn:generous-is-strong})

\item If \(\bbP\) and \(\bbQ\) are both \(\kappa\)-descending distributive, must \(\bbP \times \bbQ\) also be? (\cref{qn:product-of-generous})

\item Let \(\kappa\) be weakly inaccessible. Does `for all regular, non-zero \(\alpha<\kappa\), \(\lnot \AL_\alpha\)' imply \(\lnot \AL_\kappa\)? If not, what is the weakest large cardinal assumption on \(\kappa\) required to make this true? (\cref{qn:do-we-need-weirdness,qn:large-cardinal-continuity})

\item Does \(\SVC(S) \land \AL_\alpha(S)\) imply \(\AL_\alpha\)? (\cref{qn:improve-lrc})

\item Do \(\kappa\)-descending distributive symmetric systems preserve \(\EDC_\kappa\)? (\cref{qn:generous-adc})

\item Does \(\SVC(S)\) imply that \(\aleph^\ast(S)\) is regular? (\cref{qn:is-aleph-ast-s-regular})
\end{enumerate}

\subsection*{Acknowledgements}

The author would like to thank Asaf Karagila for many helpful discussions regarding descending distributivity and extendable choice. The author would also like to thank the set theorists of Leeds for their input on the name `descending distributive'.

\renewcommand*{\bibfont}{\normalfont\small}
\printbibliography

\end{document}